\documentclass[10pt,leqno]{amsart}
\usepackage{amssymb}
\usepackage{mathrsfs}
\usepackage{color}
\usepackage{soul}
\usepackage{graphicx}
\usepackage{enumitem}
\usepackage[normalem]{ulem}

\usepackage[pdfborder={0 0 0}]{hyperref}

\usepackage{MnSymbol}

\newcommand\N{{\mathbb N}}
\newcommand\R{{\mathbb R}}

\newcommand\C{{\mathbb C}}

\newcommand\E{{\mathbb E}}
\newcommand\Pp{{\mathbb P}}

\def\AA{{\mathcal A}}
\def\BB{{\mathcal B}}
\def\CC{{\mathcal C}}

\def\DD{{\mathcal D}}
\def\EE{{\mathcal E}}
\def\FF{{\mathcal F}}

\def\HH{{\mathcal H}}

\def\LL{{\mathcal L}}

\def\RR{{\mathcal R}}
\def\SS{{\mathcal S}}
\def\TT{{\mathcal T}}
\def\UU{{\mathcal U}}
\def\VV{{\mathcal V}}

\def\XX{{\mathcal X}}

\def\ZZ{{\mathcal Z}}

\def\BBB{{\mathscr B}}
\def\CCC{{\mathscr C}}

\def\HHH{{\mathscr H}}

\def\JJJ{{\mathscr J}}

\def\LLL{{\mathscr L}}

\DeclareMathOperator{\sign}{sign}

\def\eps{{\varepsilon}}

\newtheorem{theo}{Theorem}
\newtheorem{prop}[theo]{Proposition}
\newtheorem{lem}[theo]{Lemma}
\newtheorem{cor}[theo]{Corollary}
\newtheorem{rem}[theo]{Remark}

\newtheorem{defin}[theo]{Definition}

\newcommand{\beqn}{\begin{equation}}
\newcommand{\eeqn}{\end{equation}}
\newcommand{\bear}{\begin{eqnarray}}
\newcommand{\eear}{\end{eqnarray}}
\newcommand{\bean}{\begin{eqnarray*}}
\newcommand{\eean}{\end{eqnarray*}}
\newcommand{\bal}{\begin{aligned}}
\newcommand{\eal}{\end{aligned}}

\newcommand{\re}{\text{Re}}

\definecolor{orange}{rgb}{1.00,0.50,0.0}

\newcommand{\vertiii}[1]{{\left\vert\kern-0.25ex\left\vert\kern-0.25ex\left\vert #1 
    \right\vert\kern-0.25ex\right\vert\kern-0.25ex\right\vert}}

\def\signsm{\bigskip \begin{center} {\sc St\'ephane Mischler\par\vspace{3mm}
Universit\'e Paris-Dauphine \& IUF \par
CEREMADE, UMR CNRS 7534\par
Place du Mar\'echal de Lattre de Tassigny
75775, Paris Cedex 16\par
FRANCE\par\vspace{3mm}
e-mail:} \tt{mischler@ceremade.dauphine.fr} \end{center}}

\def\signcq{\bigskip \begin{center} {\sc Crist\'obal Qui\~ninao\par\vspace{3mm}
Universit\'e Pierre et Marie Curie \par
Laboratoire Jacques-Louis Lions, CNRS UMR 7598\par
4 place de Jussieu
F-75005, Paris\par
FRANCE\par
and Mathematical Neuroscience Team, CIRB\par 
College de France\par\vspace{3mm}
e-mail:} \tt{cristobal.quininao@college-de-france.fr}  \end{center}}

\def\signjt{\bigskip \begin{center} {\sc Jonathan Touboul\par\vspace{3mm}
College de France
Mathematical Neuroscience Team CIRB\par
11 place Marcelin-Berthelot
75005, Paris\par
FRANCE\par
and INRIA Paris-Rocquencourt, Mycenae Team\par \vspace{3mm}
e-mail:}\tt{jonathan.touboul@college-de-france.fr} \end{center}}

\begin{document}


\title[Kinetic {F}itz{H}ugh-{N}agumo system]{On a kinetic {F}itz{H}ugh-{N}agumo model of neuronal network}

\author{S. Mischler, C. Quininao, J. Touboul}

\begin{abstract}
 
We investigate existence and uniqueness of solutions of a McKean-Vlasov evolution PDE representing the macroscopic behaviour of interacting Fitzhugh-Nagumo neurons. This equation is hypoelliptic, nonlocal and has unbounded coefficients. We prove existence of a solution to the evolution equation and non trivial stationary solutions. Moreover, we demonstrate uniqueness of the stationary solution in the weakly nonlinear regime. Eventually, using a semigroup factorisation method, we show exponential nonlinear stability in the small connectivity regime. 
  
\end{abstract}

\maketitle

\begin{center} {\bf Preliminary version of \today}
\end{center}



%
%
%
%
%
	
\vspace{0.3cm}



\bigskip


\tableofcontents


\bigskip
\textbf{Keywords}: FitzHugh Nagumo,  Neuronal Network, 
long-time behaviour; stability.

 \smallskip

\textbf{AMS Subject Classification (2000)}:   35B45, 35B60, 35B65, 35K15, 35Q92, 92C17, 92B05


\section{Introduction} 
\label{sec:intro}
\setcounter{equation}{0}
\setcounter{theo}{0}

This paper undertakes the analysis of the existence and uniqueness of solutions for a mean-field equation arising in the modeling of the macroscopic activity of the brain. This equation describes the large-scale dynamics of a model of the central nervous system, taking into account the fact that it is composed of a very large number of interconnected cells that manifest highly nonlinear dynamics and are subject to noise. Non-linearities in the intrinsic dynamics of individual cells are an essential element of the neural code. Indeed, nerve cells constantly regulate their electrical potential depending on the input they receive. This regulation results from intense ionic exchanges through the cellular membranes. The modeling of these dynamics led to the development of the celebrated Hodgkin-Huxley model~\cite{hodgkin-huxley:52}, a very precise description of ion exchanges through the membrane and their effects on the cell voltage. A simplification of this model conserving the most prominent aspects of the Hodgkin-Huxley model, the Fitzhugh-Nagumo (FhN) model~\cite{fitzhugh:55,nagumo1962active}, has gained the status of canonical model of excitable cells in neuroscience. This model constitutes a very good compromise between versatility and accuracy on the one hand, and relative mathematical simplicity on the other hand. It describes the evolution of the membrane potential $v$ of the cell coupled to an auxiliary variable $x$, called the adaptation variable.
Different neurons interact through synapses that are either chemical or electrical. In the case of electrical synapses
 for instance, the evolution of the pair voltage-adaptation for a set of $n$ neurons $\{(v^i_t,x^i_t),\, 1\leq i\leq n\}$ satisfy the equations:
\begin{equation}\label{eq:FhNNet}
	\begin{cases}
			dv^i_t = \left( v^i_t\,(v^i_t-\lambda)\,(1-v^i_t)-x^i_t + \sum_{j=1}^n J_{ij} (v^i_t - v^j_t) +I_t \right)\,dt + \sigma \, dW^i_t\\
			dx^i_t = \left(-a x_t^i + b v^i_t\right)\,dt ,
	\end{cases}
\end{equation}
where the cubic nonlinearity accounts for the cell excitability, $I_t$ is the input level, $a$ and $b$ are positive constants representing timescale and coupling between the two variables, and the processes $\{(W^i_t)_{t\geq 0},\, 1\leq i\leq n\}$ are independent Brownian motions accounting for the intrinsic noise at the level of each cell. In the sequel, for sake of simplicity, we assume that $\sigma^2=2$ and $I_t=I_0\in\R$ constant, but  it is likely that some of our analysis can be extend to $I_t\in L^\infty(\R_+)$ converging rapidly when $t$ goes to infinity. The coefficients $J_{ij}$ represent the effect of the interconnection of cell $j$ onto cell $i$. These coefficients are positive, and incorporate the information of the connectivity map. Under relatively weak assumptions on the distribution of these coefficients (see Appendix~\ref{app:MFLim}), it is relatively classical to show that the system enjoys propagation of chaos property and finite sets of neurons converge in law towards a process whose density solves the McKean-Vlasov evolution PDE:
\begin{eqnarray}\label{eq:FhNs}
&&\partial_t f = Q_\eps[\JJJ_f]\,f := \partial_x(A f) + \partial_v\big(B_\eps(\JJJ_f)f\big) + \partial_{vv}^2 f  \vphantom{\int} \quad \hbox{on} \,\, (0,\infty) \times \R^2, 
\\ \label{eq:FhNs2}
&& A = A(x,v)= ax-bv, \quad B_\eps(\JJJ_f) = B(x,v\,;\eps,\JJJ_f)  \vphantom{\int} ,
\\ \label{eq:FhNs3}
&&B(x,v\,;\eps,j) = v\,(v-\lambda)\,(v-1) + x  - \eps \, (v-j) + I_0, \quad \JJJ_f = \JJJ(f) = \int_{\R^2}v \, f(x,v) \, dvdx ,
\end{eqnarray}
where $\eps$ denotes the averaged value of the connectivity coefficients $J_{ij}$ and $f = f(t,x,v) \ge 0$ is the density function of finding neurons with adaptation and voltage $(x,v) \in \R^2$ at time $t \ge 0$. The evolution equation~\eqref{eq:FhNs} is complemented by an initial condition
$$
f(0,\cdot,\cdot) = f_0(\cdot,\cdot)  \ge 0 \quad \hbox{in} \,\,  \R^2.
$$
Since the PDE can be written in divergence form, the initial normalization of the density is conserved. In particular, consistent with the derivation of the system, we have:
$$
\int_{\R^2} f(t,x,v) \, dxdv = \int_{\R^2} f_0(x,v) \, dxdv = 1,
$$
 when $f_0$ is normalized. Moreover, the nonnegativity is also a classical result of this kind of equations (for a brief discussion see Section~\ref{sec:estim}), therefore we assume in the sequel that $f$ is a probability density.


\bigskip
From the mathematical viewpoint, this equation presents several interests. First, the system is not Hamiltonian and the dynamics may present several equilibria, therefore, methods involving a potential and its possible convexity may not be used. Second, intrinsic noise acts as a stochastic input only into the voltage variable (since it modifies the voltage through random fluctuations of the current), leaving the adaptation equation unchanged and yielding to a hypoelliptic equation.
From the phenomenological viewpoint, this system is particularly rich. It shows a number of different regimes as parameters are varied, and in particular, as a function of the connectivity level: the system goes from a non-trivial stationary regime in which several stationary solutions may exist for strong coupling, to periodic solutions, and eventually to a unique stationary solution for weak coupling. This is illustrated in section~\ref{sec:Beyond}, in particular, we present some numerical results of~\eqref{eq:FhNNet} for a large number of interacting neurons.

In order to rigorously analyse equation~\eqref{eq:FhNs}, we restrict ourself to the latter regime, and we shall demonstrate the existence, uniqueness and stability of solutions to the McKean-Vlasov equation in the limit of weak coupling. 
More precisely, we shall prove existence of solution and non trivial stationary solution to the evolution equation~\eqref{eq:FhNs} without restriction on the connectivity coefficient $\eps>0$, and next uniqueness of the stationary solution and its exponential NL stability in the small excitability regime. 

\subsection{Historical overview of macroscopic and kinetic models in neuroscience}
As mentioned above, the problem we study lies within a long tradition of works in the domain of the characterization of macroscopic behaviors in large neuronal networks. First efforts to describe the macroscopic activity of large neuron ensemble can be traced back to the work of Amari, Wilson and Cowan in the 1970s~\cite{amari:72,amari:77,wilson-cowan:72,wilson-cowan:73}, where were introduced heuristically derived equations on the averaged membrane potential of a population of neurons. These models made the assumption that populations interact through a macroscopic variable, the averaged firing rate of the population, assumed to be a sigmoidal transform of the mean voltage. This model has been extremely successful in reproducing a number of macroscopic behaviors in the cortex, one of the most striking being related to pattern formation in the cortex associated to visual hallucinations~\cite{ermentrout1979mathematical}  (see also~\cite{bressloff:12} for a recent review on the subject). The relatively simplicity and good agreement with neurological phenomena motivated to understand the relationship between the dynamics of individual cells activity and macroscopic models. This has been an important piece of work in the 1990s in the bio-physics community, using simplified (non-excitable) models and specific assumptions on the architecture of the network, including the assumption of sparse and balanced connectivity (the sum of all incoming input vanishes). The sparse connectivity assumption was used by the authors to stated that the activity was uncorrelated~\cite{abbott-van-vreeswijk:93,amit-brunel:97,brunel-hakim:99}, and resulted in characterizing different neuronal states. Alternative approaches were also developed based on population density~\cite{cai-tao:04} methods. These yield complex partial differential equations, that were reduced to a set of moment equations from which authors may deduce the behavior of the system. The validity of these moment reduction and their well-posedness is a complex issue debated in the literature, see e.g.~\cite{ly-tranchina:07}. A transition Markov two-states model governing the firing dynamics of the neurons in the network was recently introduced. In these models, the transition probability of the system, written through a master equation, is then handled using different physics techniques including van Kampen expansions or path integral methods. This modeling recently gathered the interest of the community (see for example~\cite{buice-cowan:07,bressloff:09,elboustani-destexhe:09,touboul-ermentrout:11}).

The mathematical community also undertook the analysis of the problem since the beginning of this decade. In that domain, one can distinguish also two distinct approaches: on one side, the development of mathematical models for simplified or phenomenological neuronal models, and on the other side works on the precise neuronal models. The dynamics of solutions of macroscopic limits of phenomenological neuron models is much more developed. The characterization of the stationary (or periodic) solutions was done in a simplified model, the Wilson-Cowan system, which has the important advantage to yield a Gaussian solution whose mean and standard deviation satisfy a deterministic dynamical system that may be studied analytically~\cite{touboulNeuralFieldsDynamics:11,touboul-hermann-faugeras:11} using the analysis of ordinary differential equations. Artificial spiking neuronal models representing the discontinuous dynamics of the time to the next spike were analyzed in a number of situations, including construction of periodic solutions to the limit equation in the presence of delays~\cite{pakdaman-perthame-etal:10,pakdaman2012adaptation,pakdaman2013relaxation}. In the same vein, an important result was demonstrated on integrate-and-fire models in the presence of noise and excitation: it was shown that too much excitation could prevent the existence of solutions for all times, as the firing rate blows up in finite time~\cite{caceres-carrillo:11}. These approaches make use of functional analysis of PDEs and nonlocal age-structured type of equations. 

\subsection{Organization of the paper}
The paper is organized as follows. Section~\ref{sec:DefResNot} summarizes our main results that are demonstrated in the rest of the paper. Section~\ref{sec:estim} is interested with the existence, uniqueness and \emph{a priori} estimates on the solutions to the evolution equation, as well as, the existence of stationary solutions. The next sections prove the stability of the unique stationary solution. Our proof uses factorization of the linearized semigroup allowing to prove linear stability, which we complete in section~\ref{sec:StabilityStatSol} by an analysis of the nonlinear stability of the stationary solution. Along the way, a number of open problems were identified beyond the small connectivity regime treated here that we present in section~\ref{sec:Beyond} together with numerical simulations: we will observe that the stationary solution splits into two stable stationary solutions as connectivity is increased, and in an intermediate regime, periodic solutions emerge. Two appendices complete the paper. Appendix~\ref{app:MFLim} investigates the microscopic system and its convergence towards the mean-field equation~\eqref{eq:FhNs} and Appendix~\ref{app:MaximumPrinciple} deals with the strict positivity of stationary solutions.

\section{Summary of the main results}\label{sec:DefResNot}
\label{sec:summary}
\setcounter{equation}{0}
\setcounter{theo}{0}

\subsection{Functional spaces and norms} 
We start by introducing the functional framework in which we work throughout the paper. For any exponent $p\in[1,\infty]$ and any nonnegative weight function $\omega$, we denote by $L^p(\omega)$ the Lebesgue space $L^p(\R^2;\omega \,dx\,dv)$ and for $k\in\N$ the corresponding Sobolev spaces $W^{k,p}(\R^2;\omega \,dx\,dv)$. They are associated to the norms
$$
\| f \|_{L^p(\omega)} = \| f \omega \|_{L^p},\quad
\| f \|_{W^{k,p}(\omega)}^p= \|f\|^p_{L^p(\omega)}+\sum_{j=1}^k \|D_{x,v}^k f \|_{L^p(\omega)}^p.
$$
For $k\geq1$, we define the partial $v$-derivative space $W^k_v(\omega)$ by
$$
 W^{k,p}_v(\omega)\,:=\,\{\,f\,\in\,W^{k-1,p}(\omega)\,;\, D_v^kf\,\in\, L^p(\omega) \,\},
$$
and it is natural to associate them to the norm
$$
 \| f \|_{W^{k,p}_v(\omega)}^p= \| f \|_{W^{k-1,p}(\omega)}^p+\|D^k_v f\|_{L^p(\omega)}^p.
$$
A particularly important space in our analysis, denoted by $H^2_v(\omega)$, is
$$
H^2_v(\omega)= W^{2,2}_v(\omega) = \{f\in H^1(\omega)\text{ such that }\partial_{vv}^2 f\in L^2(\omega)\},
$$
together with the set of functions with finite entropy
$$
 L^1\,\log L^1\,\,:=\,\,\Big\{f\,\in\,L^1(\R^2)\text{ such that } f\geq0\text{ and }\HHH(f)<\infty\Big\},
$$
where we use the classical notation $\HHH(f):=\int_{\R^2} f\,\log f$.
Finally, for $\kappa>0$, let us introduce the exponential weight function:
\begin{equation}\label{eq:weightFunc}
m=e^{\kappa (M-1)}\quad\text{with}\quad M:= 1+x^2/2 + v^2/2.
\end{equation}
In the sequel, we will be brought to vary the constant $\kappa$ involved in the definition of $m$, therefore we introduce the shorthand $m_i=e^{\kappa_i (M-1)}, \; i\in\N$. Unless otherwise specified, these sequences are constructed under the assumption that the sequence $\kappa_i$ is strictly increasing. 
 
\subsection{Main results} 

We start by stating a result related to the well possedness of~\eqref{eq:FhNs} and to the \emph{a priori} bounds on the solution. Using classical theory of renormalized solutions, it is not hard to see that equation~\eqref{eq:FhNs} has indeed {\it weak solutions}, which we naturally define as:
\begin{defin}\label{def:One}
Let $f_0$ be a normalized nonnegative function defined on $\R^2$ such that $\JJJ(f_0)$ is well defined. We say that $f_t(x,v)\,:=\, (t,x,v)\mapsto f(t,x,v)$ is a weak solution to~\eqref{eq:FhNs} if the following conditions are fulfilled:
\begin{itemize}
 \item[-] $f\in C([0,\infty); L^1(M^2))$;
 \item[-] for almost any $t\geq0$, $f\geq0$ and
 $$
  \int_{\R^2} f(t,x,v)\,dx\,dv \,=\, \int_{\R^2} f_0(x,v)\,dxdv \,=\,1;
 $$
 \item[-] for any $\varphi\in C^1([0,\infty);C^\infty_c(\R^2))$ and any $t\geq0$ it holds
\begin{equation}\label{eq:Solution}
 \int_{\R^2} \varphi f_t = \int_{\R^2}\varphi f_0 +\int_0^t\int_{\R^2} \big[\partial_t\varphi+\partial^2_{vv}\varphi-A\,\partial_x\varphi-B_\eps(\JJJ(f_s))\partial_v\varphi\big] f_s.
\end{equation}
\end{itemize}
\end{defin}
\bigskip 

\noindent Equipped with this definition we can state the
\begin{theo} \label{th:E&U&B} For any $f_0 \in L^1(M^2) \cap L^1 \log L^1\cap \Pp(\R^2)$, there exists a unique {\it global weak solution} $f_t$ to the FhN equation~\eqref{eq:FhNs}, that moreover satisfies
\beqn\label{eq:L1H}
 \|f_t\|_{L^1(M)} \le \max ( C_0, \| f_0 \|_{L^1(M)} ),
\eeqn
and depends continuously in $L^1(M)$ to the initial datum. More precisely, if $f_{n,0} \to f_0$ in $L^1(M)$ and $\HHH(f_{n,0})\le C$ then $f_{n,t} \to f_t$ in $L^1(M)$ for any later time $t \ge 0$. 
 
Furthermore, there exist two norms $\|\cdot\|_{\HH^1}$ and $\|\cdot\|_{\HH^2_v}$ equivalent respectively to $\|\cdot\|_{H^1(m)}$ and $\|\cdot\|_{H^2_v(m)}$, such that the following estimates hold true:
\beqn\label{eq:L1m}
 \| f_t\|_{L^1(m)} \le \max ( C_1,  \|f_0\|_{L^1(m)} ),
 \eeqn
 as well as
 \beqn\label{eq:H1m}
\| f_t \|_{\HH^1} \le  \max ( C_2,  \| f_0 \|_{\HH^1} ) ,
\eeqn
and
 \beqn\label{eq:H2vm}
 \|f_t\|_{\HH^2_v}\leq \max ( C_3,\|f_0\|_{\HH^2_v} ),
 \eeqn
where $C_1,C_2,C_3$ are positive constants.
 
\end{theo}

\smallskip
 \noindent The other two main results of the present work can be summarized in the following
\begin{theo} \label{th:StatExist} For any $\eps\geq0$, there exists at least one stationary solution $G_\eps$ to the  FhN statistical equation \eqref{eq:FhNs}, that is 
\beqn\label{eq:StatExist}
G_\eps \in H^2_v(m) \cap \Pp(\R^2), \quad
0 =  \partial_x(A G_\eps) + \partial_v(B_\eps(\JJJ_{G_\eps}) G_\eps) + \partial_{vv}^2 G_\eps \quad \hbox{in} \quad  \R^2.
\eeqn
Moreover, 
there exists an increasing function $\eta : \R_+ \to \R$ such that $\eta(\eps)\xrightarrow[\eps\rightarrow0]{} 0$ and such that any solution to \eqref{eq:StatExist}  satisfies 
$$
\|G  - G_0 \|_{L^2(m)} \le \eta(\eps),
$$
where $G_0$ is the unique stationary solution corresponding to the case $\eps=0$.
 \end{theo}
\smallskip

\begin{theo} \label{th:WeakCregime} There exists $\eps^* > 0$ such that, in the small connectivity regime $\eps \in (0,\eps^*)$, the stationary solution 
is unique and exponentially stable. More precisely, there exist $\alpha^*< 0$ and $\eta^*(\eps):\R_+\rightarrow\R$, with $\eta^*(\eps)\xrightarrow[\eps\rightarrow0]{}\infty$, such that if
$$
f_0 \in H^1(m) \cap \Pp(\R^2)\quad \text{and}\quad\|f_0 - G \|_{H^1(m)} \le \eta^*(\eps),
$$
then there exists $C^*=C^*(f_0,\eps^*,\eps)>0$, such that
$$
\|f_t - G \|_{L^2(m)} \le C^* \, e^{\alpha^*\,t}, \quad  \forall \, t \ge 0,
$$
where $f_t$ is the solution to~\eqref{eq:FhNs} with initial condition $f_0$.
\end{theo}
\bigskip

\subsection{Other notations and definitions.}

We prepare to the demonstration of these results by introducing a few notations that will be used throughout the paper. For two given Banach spaces $(E,\|\cdot\|_E)$ and $(\EE,\|\cdot\|_\EE)$, we denote by $\BBB(E,\EE)$ the space of bounded linear operators from $E$ to $\EE$ and we denote by $\|\cdot\|_{\BBB(E,\EE)}$ the associated operator norm. The set of closed unbounded linear operators from $E$ to $\EE$ with dense domain is denoted by $\CCC(E,\EE)$. In the special case when $E=\EE$, we simply write $\BBB(E)=\BBB(E,E)$ and $\CCC(E)=\CCC(E,E)$.
\smallskip

For a given $\alpha\in\R$, we define the complex half plane
$$
 \Delta_\alpha:=\{z\in\C,\quad \re(z)>\alpha\}.
$$
For a given Banach space $X$ and $\Lambda\in\CCC(X)$ which generates a semigroup, we denote this associated semigroup by $(S_\Lambda(t),\,t\geq0)$, by $D(\Lambda)$ its domain, by $N(\Lambda)$ its null space, by $R(\Lambda)$ its range, and by $\Sigma(\Lambda)$ its spectrum. On the resolvent set $\rho(\Lambda)=\C\setminus\Sigma(\Lambda)$ we may define the resolvent operator $\RR(\Lambda)$ by 
$$
 \forall\,z\in\C,\qquad \RR_\Lambda(z):=(\Lambda-z)^{-1}.
$$
Moreover, $\RR_\Lambda(z)\in\BBB(X)$ and has range equal to $D(\Lambda)$. We recall that $\xi\in\Sigma(\Lambda)$ is called an eigenvalue of $\Lambda$ if $N(\Lambda-\xi)\neq\{0\}$, and it called an isolated eigenvalue if there exists $r>0$ such that
$$
 \Sigma(\Lambda)\cap\{z\in\C,\,|z-\xi|< r\}=\{\xi\}.
$$

Since the notion of convolution of semigroups will be required, we recall it here. Let us consider some Banach spaces $X_1,X_2$ and $X_3$ and two given functions
$$
 S_1\in L^1([0,\infty);\BBB(X_1,X_2))\quad\text{and}\quad S_2\in L^1([0,\infty);\BBB(X_2,X_3)),
$$
one can define $S_2\ast S_1\in L^1([0,\infty);\BBB(X_1,X_3))$ by
$$
 (S_2\ast S_1)(t):=\int_0^t S_2(t-s)S_1(t)\,ds,\qquad \forall\,t\geq0.
$$
In the special case $S_1=S_2$ and $X_1=X_2=X_3$, $S^{(*n)}$ is defined recursively by $S^{(\ast 1)}=S$ and $S^{(\ast n)}=S\ast S^{(\ast(n-1))}$ for $n>1$. Equipped with this definition, we state the
\begin{prop}\label{prop:semigroupGeneral}
 Let $X,Y$ be two Banach spaces such that $Y\subset X$. Let us consider $S(t)$ a continuous semigroup such that for all $t\geq0$
 \begin{equation}\label{eq:genHypo}\nonumber
  \|S(t)\|_{\BBB(\XX)}\leq C_{\XX}\, e^{\alpha^* t},\quad \XX\in\{X,Y\},
 \end{equation}
 for some $\alpha^*\in\R$ and positive constants $C_X$ and $C_Y$. If there exists $\Theta>0$ and $C_{X,Y}>0$ such that
 \begin{equation}\label{eq:genReg}\nonumber
  \|S(t) f\|_Y\,\leq\, C_{X,Y}\,t^{-\Theta}\,e^{\alpha^*t}\,\|f\|_X ,\quad \forall\, f\in  X,\,t\in(0,1],
 \end{equation}
 then, there exists $n\in\N$, and a polynomial $p_n(t)$ such that
 \begin{equation}\label{eq:convSn}
  \|S\,^{(*n)}(t) f\|_Y\,\leq\, p_n(t)\, e^{\alpha^*t}\|f\|_X,\qquad\forall\, f\in X,\,t>0.
 \end{equation}
 In particular, for any $\alpha>\alpha^*$, it holds
 $$
  \|S\,^{(*n)}(t) f\|_Y\,\leq\, C_{\alpha,n} e^{\alpha t}\|f\|_X,\quad\forall\, f\in X,\,t>0,
 $$
 for some positive constant $C_{\alpha,n}$.
\end{prop}

This general result has been already established and used in~\cite{GMM} and~\cite{MM*}, but we give an alternative, and somehow simpler, proof of it.

\begin{proof}
Let us start by noticing that for $\XX\in\{X,Y\}$, if
\begin{equation}\label{eq:inducHyp}
 \|S\,^{(*n)}(t) f\|_{\XX} \,\leq\, p_n^\XX(t)\,e^{\alpha^*t}\|f\|_{\XX},\quad\forall\,t\geq0,
\end{equation}
for $n\in\N$ and $p_n^\XX(t)$ a polynomial, then
$$
 \|S\,^{(*(n+1))}(t) f\|_{\XX}
 \,\leq\,\int_0^t \|S(t-s)\,S\,^{(*n)}\,(s) f\|_{\XX}\,ds
  \,\leq\,p_{n+1}^\XX(t)\,e^{\alpha^*t}\|f\|_{\XX},
$$
for $p_{n+1}^\XX=C_\XX\int_0^tp_n^\XX(s)\,ds$. So, by an immediate induction argument we get~\eqref{eq:inducHyp} for any $n\geq1$ and $p_n^\XX(t):=\frac{C_\XX^n t^{n-1}}{(n-1)!}$.
\smallskip

Let us now fix $t\in(0,1]$ and, without lost of generality, assume that $\Theta\notin\N$. In that case, if
\begin{equation}\label{eq:regnsmall}
 \|S\,^{(*n)}(t)f\|_{Y}\leq C_nt^{-(\Theta-n+1)}e^{\alpha^*t}\|f\|_{X},\quad\forall\,t\in(0,1],
\end{equation}
for some $n\in\N$ and $C_n$ a positive constant, then
\begin{eqnarray*}
 \|S^{(\ast (n+1))}(t) f\|_{Y} &\leq& \int_0^{t/2} \|S(t-s)S^{(\ast n)}(s) f\|_{Y}\,ds+\int_{t/2}^t \|S(t-s)S^{(\ast n)}(s) f\|_{Y}\,ds
 \\
 &\leq& \int_0^{t/2} C_{X,Y}(t-s)^{-\Theta}e^{\alpha^*(t-s)}\|S^{(\ast n)}(s) f\|_{X}\,ds+\int_{t/2}^t C_Ye^{\alpha^*(t-s)}\|S^{(\ast n)}(s) f\|_{Y}\,ds
 \\
  &\leq& \int_0^{t/2} C_{X,Y}(t-s)^{-\Theta}e^{\alpha^*t} p_n^X(s)\|f\|_{X}\,ds +\int_{t/2}^t C_Ye^{\alpha^*t} C_n s^{-(\Theta-n+1)}\| f\|_{X}\,ds
 \\
  &\leq& \frac{C_{X,Y}C_{X}^n}{(n-1)!}e^{\alpha^*t}\|f\|_{X}\int_0^{t/2} (t-s)^{-\Theta} s^{n-1}\,ds +C_YC_ne^{\alpha^*t}\| f\|_{X}\int_{t/2}^t  s^{-(\Theta-n+1)}\,ds
  \\
  &\leq& C_{n+1}t^{-(\Theta-n)}e^{\alpha^* t}\|f\|_X,
\end{eqnarray*}
for some $C_{n+1}$ depending only on $C_X,C_Y,C_{X,Y}$ and $C_n$. Once again, by an induction argument, we get~\eqref{eq:regnsmall}. Moreover, as soon as $\Theta-n+1>0$, inequality~\eqref{eq:convSn} holds.

Finally, to get the conclusion in the case $t>1$, it suffices to notice that
$$
 \|S(t)f\|_{Y}\,\,\leq\,\, C_YC_{X,Y}(t-\lfloor t\rfloor)^{-\Theta} e^{\alpha^*t}\|f\|_X,
$$
where $\lfloor t\rfloor$ is the largest integer smaller than $t$. A similar argument that the one used for $t\in(0,1]$, allows us to find a polynomial $p_n$ such that~\eqref{eq:convSn} still holds when $t>1$.
\end{proof}
\smallskip

Finally, we recall the abstract notion of \emph{hypodissipative operators}:
\begin{defin}
 Considering a Banach space $(X,\|\cdot\|_X)$, a real number $\alpha\in\R$ and an operator $\Lambda\in\CCC(X)$, $(\Lambda-\alpha)$ is said to be hypodissipative on $X$ if there exists some norm $\vertiii{\cdot}_X$ on $X$ equivalent to the usual norm $\|\cdot\|_X$ such that
$$
 \forall\,f\in D(\Lambda),\quad\exists\, \phi\in F(f)\quad\text{such that}\quad \re\langle\phi,(\Lambda-\alpha)f\rangle\leq0,
$$
where $\langle\cdot,\cdot\rangle$ is the duality bracket in $X$ and $X^*$ and $F(f)\subset X^*$ is the dual set of $f$ defined by
$$
 F(f)=F_{\vertiii{\cdot}_X}(f):=\{\phi\in X^*,\quad\langle \phi,f\rangle=\vertiii{f}_X^2=\vertiii{\phi}_{X^*}^2\}.
$$
\end{defin}
One classically sees (we refer to for example~\cite[Subsection 2.3]{GMM}) that when $\Lambda$ is the generator of a semigroup $S_\Lambda$, for given $\alpha\in\R$ and $C>0$ constants, the following assertions are equivalent:
\begin{itemize}
\item[(a)] $(\Lambda-\alpha)$ is hypodissipative;
\item[(b)] the semigroup satisfies the growth estimate $\|S_\Lambda (t)\|_{\BBB(X)}\leq C e^{\alpha t},$ $t\geq0$;
\end{itemize}

%
%
%

\section{Analysis of the nonlinear evolution equation} 
\label{sec:estim}
\setcounter{equation}{0}
\setcounter{theo}{0}

This section is concerned with the analysis of the nonlinear evolution equation. We shall prove existence and uniqueness of solutions, and provide some \emph{a priori} estimates on their behavior. 

Before going into further details, let us remark that for $\JJJ$ fixed, the operator $Q_\eps[\JJJ]$ is linear and writes
$$
 Q_\eps[\JJJ]\, f \,=\, \partial_x(Af)+\partial_v(B_\eps(\JJJ)\,f)+\partial^2_{vv}f.
$$
In particular, for $g\in H^2_v(m)$ we have
$$
 \int_{\R^2} (Q_\eps[\JJJ]\, f) \, g \,dvdx \,=\, -\int_{\R^2} f\, \big(A\,\partial_x g + B_\eps(\JJJ)\,\partial_v g-\partial^2_{vv} g\big) \,dvdx,
$$
therefore, it is natural to define
$$
 Q_\eps^*[\JJJ]\,g \, := \, -A\,\partial_x g - B_\eps(\JJJ)\,\partial_v g+\partial^2_{vv} g.
$$

\subsection{A priori bounds.}
We now fix $\eps_0>0$. The a priori estimates that follow are uniform in $\eps$ in the bounded connectivity regime $\eps\in[0,\eps_0)$, i.e., they involve constants that do not depend on $\eps$.

\begin{lem}\label{lemma:L1M} For $f_t$ solution to~\eqref{eq:FhNs} with $f_0 \in L^1(M) \cap \Pp(\R^2)$, estimate \eqref{eq:L1H} holds. Moreover, there exists $C_0'>0$ depending on $a,b,\lambda,I_0,\eps_0$ and $\|f_0\|_{L^1(M)}$ such that
\beqn\label{eq:mufbdd}
\sup_{t\geq0}|\JJJ(f _t)|<C_0'.
\eeqn
\end{lem}

\begin{proof}
We first apply Cauchy-Schwartz's inequality to find
 \begin{equation}\label{aux1}
 |\JJJ(f)| \le \int_{\R^2} |v| \, f \le\Bigl(\int_{\R^2} f\Bigr)^{1/2}\Bigl(\int_{\R^2}v^2 \,  f\Bigr)^{1/2}=\Bigl(\int_{\R^2}v^2 \,  f\Bigr)^{1/2},
 \end{equation}
 for any $f\in\Pp(\R^2)\cap L^1(v^2)$. Now, for $f_t$ a solution to~\eqref{eq:FhNs}, we have
  \begin{eqnarray*}
  \frac{d}{dt}\int_{\R^2} f_t \,M &=& \int_{\R^2} (Q_\eps[\JJJ_{f_t}]\,f_t)\,M =\int_{\R^2}f_t \,(Q_\eps^*[\JJJ_{f_t}]M)\\
  &=&\int_{\R^2}(1-Ax-B_\eps(\JJJ_{f_t})v)f_t.
 \end{eqnarray*}
  Using the definition of $A$ and $B_\eps$, and then~\eqref{aux1}, we get
 \begin{eqnarray*}
  \frac{d}{dt}\int_{\R^2}f_t M\,dxdv &\leq&-\int_{\R^2}\Big(-1+ax^2-bxv+v^2(v-\lambda)(v-1)-\eps v^2+xv+I_0v\Big)f_t +\eps \JJJ(f_t)^2
  \\
  &\leq&K_1-K_2\int_{\R^2}(v^4+x^2)f_t +\eps \int_{\R^2}v^2\,f_t
  \\
  &\leq& K_1 -  K_2 \int_{\R^2}f_t \, M\,dxdv,
 \end{eqnarray*}
where $K_1$ and $K_2$ are generic constans depending only on $a,b,\lambda,I_0$ and $\eps_0$. Using Gronwall's lemma we get~\eqref{eq:L1H} for some $C_0>0$. Finally, coming back to~\eqref{aux1}, we get
\begin{equation}\nonumber
 |\JJJ(f_t)|^2 \,\,\leq\,\, \int_{\R^2}v^2 \,  f_t \,\,\leq\,\, 2\,\|f_t\|_{L^1(M)}\,\leq\,\, 2\,\max ( C_0, \| f_0 \|_{L^1(M)} ),
 \end{equation}
which is nothing but~\eqref{eq:mufbdd}. 
\end{proof}

\begin{lem} For any $\JJJ \in \R$ fixed, there exist some constants $K_1,K_2>0$ depending on $a,b,\lambda,I_0,\JJJ,\kappa$ and $\eps_0$ such that
\beqn\label{eq:BdissipL1m}
\int_{\R^2} Q_\eps[\JJJ]\,f\cdot \sign (f)\,m \,\le\, K_1\|f\|_{L^1(\R^2)} - K_2\| f \|_{L^1(m)}, \quad\forall\,f \in L^1(m).
\eeqn
\end{lem}
\begin{proof}
Since $\JJJ\in\R$ is now fixed, for simplicity of notation, we drop the dependence on this parameter. Using Kato's inequality, $\sign(f)\partial_{vv}^2 f\leq\partial_{vv}^2|f|$, we have
\begin{eqnarray*}
 \int_{\R^2}Q_\eps\,f\cdot \sign (f)\,m &\leq& \int_{\R^2}|f|\,Q_\eps^*\,m\\
 &=& -\kappa\int_{\R^2}|f|\,\big(Ax+B_\eps v-(1+\kappa v^2)\big)m,
\end{eqnarray*}
thus
\begin{eqnarray*}
\int_{\R^2}Q_\eps\,f\cdot \sign (f)\,m \,\leq\,-\int_{\R^2}p(x,v)\,|f|\,m,
\end{eqnarray*}
 where $p(x,v)$ is a polynomial on $x$ and $v$ with leading term $v^4+x^2$. Inequality~\eqref{eq:BdissipL1m} follows directly.
 \end{proof}
  
\begin{cor} Estimate \eqref{eq:L1m} holds.
\end{cor}
\begin{proof}
 For $f_t$ solution to~\eqref{eq:FhNs}, inequality~\eqref{eq:mufbdd} tells us that $|\JJJ(f_t)|\leq C_0'$.  Moreover, since the mass is unitary for almost any $t\geq0$, it holds
 $$
  \frac{d}{dt}\int_{\R^2}|f_t|\,m \,=\, \int_{\R^2}Q_\eps[\JJJ_{f_t}]\,f_t\cdot\sign(f_t)\,m \,\leq\, K_1-K_2 \int_{\R^2}|f_t|\,m,
 $$
 where $K_1$ and $K_2$ depend only on $a,b,\lambda,I,\eps_0$ and $C_0'$. Finally, integrating this last inequality, we get
 $$
  \|f_t\|_{L^1(m)}\,\,\leq\,\,\max\big(C_1,\|f_0\|_{L^1(m)}\big),\qquad\forall\, t\geq0,
 $$
 for some positive constant $C_1$ depending only on the parameters of the system, $\eps_0$ and $C_0'$.
\end{proof}

Now we analyse the $H^1(m)$ and $H^2_v(m)$ norms of the solutions to~\eqref{eq:FhNs}, in particular, we prove a priori bounds~\eqref{eq:H1m} and~\eqref{eq:H2vm}. Since the equation is hypodissipative, we used the ideas of ``twisted spaces'' and the Nash-Villani's technique (see e.g.~\cite{MR2562709}) to control the $L^2(\R^2)$ contributions in function of the $L^1(\R^2)$ norm.

\begin{lem} For $0<\kappa_1<\kappa_2$, let us consider two exponential weight functions $m_1$ and $m_2$ as defined in~\eqref{eq:weightFunc}. For any $\JJJ \in \R$ fixed, there exist $K_1,K_2>0$ and $\delta\in(0,1)$ constants such that
\begin{equation}\label{eq:BdissipH1m}
\langle Q_\eps[\JJJ]\,f,  f \rangle_{\HH^1} \le K_1\|f\|_{L^2(\R^2)}^2 - K_2\| f \|_{\HH^1}^2,\quad  \forall\,f \in H^1(m_2),
\end{equation}
where $\langle\cdot,\cdot\rangle_{\HH^1}$ is the scalar product related to the Hilbert norm
$$
\| f \|_{\HH^1}^2 := \| f \|_{L^2(m_2)}^2 +  \delta^{3/2}\| \partial_x f \|_{L^2(m_2)}^2  + \delta^{4/3} \langle\partial_x f,\partial_v f\rangle_{L^2(m_1)}  + \delta\,   \| \partial_v f \|_{L^2(m_2)}^2.
$$ 
\end{lem}

\begin{rem}
It is worth emphasising that for $\delta\in(0,1)$ the norm $\HH^1$ is equivalent to the usual norm of $H^1(m_2)$. Indeed, the choice of the exponents allows us to write
\begin{eqnarray*}
 c_\delta\|f\|^2_{H^1(m_2)}\leq\|f\|^2_{L^2(m_2)}+\Big(\delta^{3/2}-\frac{\delta^{5/3}}{2}\Big)\|\partial_xf\|^2_{L^2(m_2)}+\frac\delta2\|\partial_vf\|^2_{L^2(m_2)}\leq \|f\|^2_{\HH^1},
\end{eqnarray*}
for some $c_\delta>0$.
\end{rem}

\begin{proof}
The proof is presented as follows: the first three steps deal with inequalities in $L^2$ for $f$ and its derivatives, while the last one combines these inequalities to control the $\HH^1$ norm. Some long and tedious calculations are only outlined for the sake of clarity. In the following we denote by $k_0,k_1$ and $k_2$ some unspecified constants and drop the dependance on $\JJJ$.
\bigskip

\noindent\textit{Step 1. $L^2(m_2)$ norm.} We start by noticing that
\begin{eqnarray*}
 \langle \partial^2_{vv}f,f\rangle_{L^2(m_2)} &=&
 -\int_{\R^2}(\partial_vf)^2m_2^2+\kappa_2\int_{\R^2}(1+2\kappa_2 v^2)f^2m_2^2.
 \\
 \langle \partial_x(A f), f \rangle_{L^2(m_2)} 
 &=&\frac12\int_{\R^2}\Big[\partial_xA-A\frac{\partial_xm_2^2}{m_2^2}\Big]f^2\,m_2^2\,\,=\,\,\frac12\int_{\R^2}[a-2\kappa_2 x(ax-bv)]f^2\,m_2^2,
\end{eqnarray*}
and similarly
$$
 \langle \partial_v(B_\eps f),f \rangle_{L^2(m_2)} \,\,=\,\,\frac12\int_{\R^2}\big[3v^2-2(1+\lambda)v+\lambda-\eps-2\kappa_2 v\,B_\eps\big]f^2 \, m_2^2.
$$
Therefore, we get
\begin{eqnarray}
 \label{eq:preOne}
 \langle Q_\eps f,  f \rangle_{L^2(m_2)} =-\int_{\R^2}p(x,v)f^2\,m_2^2-\|\partial_vf\|_{L^2(m_2)}^2,
\end{eqnarray}
 where $p(x,v)$ is a polynomial in $x$ and $v$ with leading term $v^4+x^2$. In particular, there exist some positive constants $k_1$  and $k_2$ such that
\begin{equation}\label{eq:one}
 \langle Q_\eps f,  f \rangle_{L^2(m_2)} \leq k_1\| f\|^2_{L^2(\R^2)}-k_2\|f\|^2_{L^2(M^{1/2}m_2)}-\|\partial_vf\|_{L^2(m_2)}^2.
\end{equation}
\smallskip
 
\noindent\textit{Step 2. $x$-derivative bound.} We have
 \begin{eqnarray*}
 \langle \partial_x(\partial_x(A f)), \partial_xf \rangle_{L^2(m_2)} 
 &=&\frac12\int_{\R^2}\Big[3\,\partial_xA-A\frac{\partial_xm_2^2}{m_2^2}\Big](\partial_xf)^2\,m_2^2\\
 &=&\frac12\int_{\R^2}\big[3\,a-2\kappa_2 x(ax-bv)\big](\partial_xf)^2\,m_2^2,
\end{eqnarray*}
and
$$
\langle \partial_x(\partial_v(B_\eps f)), \partial_xf \rangle_{L^2(m_2)} =\int_{\R^2}[\partial_vB_\eps\partial_xf+\partial_xB_\eps\partial_vf+B_\eps\partial^2_{xv}f]\partial_xf\,m_2^2.
$$
Since $\partial_xB_\eps=1$, and observing that
$$
\int_{\R^2}[\partial_vB_\eps\partial_xf+B_\eps\partial^2_{xv}f ]\partial_xf\,m_2^2
\,=\,
\frac12\int_{\R^2}\Big[\partial_vB_\eps-B_\eps\frac{\partial_v m_2^2}{m_2^2}\Big](\partial_x f)^2m_2^2,
$$
we get
$$
\langle \partial_x(\partial_v(B_\eps f)), \partial_xf \rangle_{L^2(m_2)} \,\leq\, \int_{\R^2}|\partial_x f|\,|\partial_v f|m_2^2+
\frac12\int_{\R^2}\Big[\partial_vB_\eps-B_\eps\frac{\partial_v m_2^2}{m_2^2}\Big](\partial_x f)^2m_2^2.
$$

Using that
$$
 \langle \partial_x\partial^2_{vv}f,\partial_xf \rangle_{L^2(m_2)} =-\int_{\R^2}|\partial^2_{xv}f|^2m_2^2+\frac12\int_{\R^2}(\partial_xf)^2\partial^2_{vv}m_2^2.
$$
we finally obtain
\begin{eqnarray}\label{eq:two}
 \langle \partial_x(Q_\eps f), \partial_x f \rangle_{L^2(m_2)}
 &\leq&
 k_1\|\partial_x f\|^2_{L^2(\R^2)}-k_2\|\partial_xf\|^2_{L^2(M^{1/2}m_2)}\\
 &&
 \qquad-\|\partial^2_{xv}f\|^2_{L^2(m_2)}+\int_{\R^2}|\partial_xf||\partial_vf|\,m_2^2.\nonumber
\end{eqnarray}
 A similar calculation leads to
 \begin{eqnarray}\label{eq:three}
 \langle (\partial_vQ_\eps f), \partial_v f \rangle_{L^2(m_2)}
 &\leq&
 k_1\|\partial_v f\|^2_{L^2(\R^2)}-k_2\|\partial_vf\|^2_{L^2(M^{1/2}m_2)}\\
 &&
 -\|\partial^2_{vv}f\|^2_{L^2(m_2)}+b\int_{\R^2}|\partial_xf||\partial_vf|\,m_2^2\nonumber\\
 &&+2\kappa_2(1+\lambda)\int_{\R^2}v\,f^2m_2^2\nonumber.
\end{eqnarray}
 \textit{Step 3. Cross product bound.}
The contribution of the cross product term is a little bit more delicate. We decompose it into five quantities and we study them separately:
\begin{eqnarray*}
&&\langle\partial_x Q_\eps f, \partial_v f \rangle_{L^2(m_1)}+\langle\partial_v Q_\eps f, \partial_x f \rangle_{L^2(m_1)}\\
&&\qquad\qquad\qquad=\int_{\R^2}\big[(\partial_xf)(\partial_{vvv}^3f)+(\partial_vf)(\partial_{xvv}^3f)\big] \,m_1^2\\
&&\qquad\qquad\qquad+\int_{\R^2}\big[\partial_xA\partial_vf+\partial_vA\partial_xf+A\partial_{vx}^2f\big](\partial_x f)\,m_1^2\\
&&\qquad\qquad\qquad+\int_{\R^2}\big[\partial_{vv}^2Bf+2\partial_vB\partial_vf+B\partial_{vv}^2f\big](\partial_x f)\,m_1^2\\
&&\qquad\qquad\qquad+\int_{\R^2}\big[2\partial_xA\partial_xf+A\partial_{xx}^2f\big](\partial_v f)\,m_1^2\\
&&\qquad\qquad\qquad+\int_{\R^2}\big[\partial_vB\partial_xf+\partial_xB\partial_vf+B\partial_{xv}^2f\big](\partial_v f)\,m_1^2\quad=:\quad\sum_{i=1}^5\TT_i.
\end{eqnarray*}

We start by handling the first term on the right hand side. Using integration by parts adequately, we get
\begin{equation}\nonumber
 \TT_1 \,=\,
 \int_{\R^2} (\partial_xf)(\partial_v f)\partial_{vv}^2 m_1^2-2\int_{\R^2} (\partial_{xv}^2f)(\partial_{vv}^2f)\,m_1^2.
\end{equation}
Similarly, for the contributions involving $A$, we have
\begin{equation}\nonumber
\TT_2\,=\,\frac12\int_{\R^2}\Big[\partial_vA-A\frac{\partial_v m_1^2}{m_1^2}\Big](\partial_xf)^2m_1^2+a\int_{\R^2}(\partial_xf)(\partial_vf)\,m_1^2,
\end{equation}
and
$$
\TT_4
 \,=\,\int_{\R^2}\Big[\partial_xA-A\frac{\partial_xm_1^2}{m_1^2}\Big](\partial_xf)(\partial_vf)m_1^2+\frac12\int_{\R^2} \partial_v[Am_1^2](\partial_xf)^2.
$$
Adding these last two expressions, it only remains
\begin{equation}\nonumber
 \int_{\R^2}\partial_v A(\partial_xf)^2 m_1^2+\int_{\R^2}\Big[2a-A\frac{\partial_x m_1^2}{m_1^2}\Big](\partial_xf)(\partial_vf) m_1^2\leq -b\,\|\partial_xf\|_{L^2( m_1)}^2+k_0\int_{\R^2} |\partial_xf|\,|\partial_vf|\,M\, m_1^2,
\end{equation}
for some constant $k_0>0$.
\smallskip

For the contributions related to $B_\eps$, involved in $\TT_3$ and $\TT_5$, we have
$$
\TT_3\,\,=\,\,-\int_{\R^2}2\kappa_1 x(3v-1-\lambda)f^2m_1^2+2\int_{\R^2}\partial_vB_\eps(\partial_xf)(\partial_v f)m_1^2+\int_{\R^2}B_\eps(\partial^2_{vv}f)(\partial_xf)m_1^2,
$$
and
$$
\TT_5\,\,=\,\,\int_{\R^2}\partial_vB_\eps\,(\partial_xf)(\partial_v f)\,m_1^2+\frac12\int_{\R^2} \Big[\partial_xB_\eps-B_\eps\frac{\partial_xm_1^2}{m_1^2}\Big](\partial_vf)^2 m_1^2,
$$
Finally, for the last contribution in $\TT_3$, we have
$$\int_{\R^2} B_\eps(\partial_{vv}^2f)(\partial_x f)\, m_1^2\,\leq\, k_0\int_{\R^2} (\partial_{vv}^2f)(\partial_x f)\,M^{3/2}\,m_1^2,
$$ 
getting that there exists $k_0>0$ such that
\begin{eqnarray}\label{eq:four}
&&\langle\partial_x Q_\eps f, \partial_v f \rangle_{L^2( m_1)}
+\langle\partial_v Q_\eps f, \partial_x f \rangle_{L^2(m_1)}\\ \nonumber
&&\qquad\qquad\leq 
k_0\int_{\R^2}|\partial_xf|\,|\partial_vf| \,M\, m_1^2+k_0\int_{\R^2}|\partial_{xv}^2f|\,|\partial_{vv}^2f|\,m_1^2\\ \nonumber
&&\qquad\qquad\qquad-b\,\|\partial_xf\|^2_{L^2(m_1)}+k_0\int_{\R^2}|\partial_{vv}^2f|\,|\partial_xf|\,M^{3/2}\, m_1^2 \nonumber\\
&&\qquad\qquad\qquad\qquad+k_0\int_{\R^2}|\partial_vf|^2\,M^2\,m_1^2+k_0\int_{\R^2} f^2\,M\,m_1^2.\nonumber
\end{eqnarray}
\smallskip

\noindent\textit{Step 4. Conclusion.}
To get~\eqref{eq:BdissipH1m}, we just put together~\eqref{eq:one},~\eqref{eq:two},~\eqref{eq:three} and~\eqref{eq:four} and we use Young's inequality several times. Indeed, the scalar product $\langle\cdot,\cdot\rangle_{\HH^1}$ applied to any $f\in H^1(m_2)$ writes
\begin{eqnarray*}
 \langle Q_\eps f,f\rangle_{\HH^1} &=& \langle Q_\eps f,f\rangle_{L^2(m_2)}\\
 &&\quad+\delta^{3/2}\langle\partial_xQ_\eps f,\partial_xf\rangle_{L^2(m_2)}+\delta\,\langle\partial_vQ_\eps f,\partial_vf\rangle_{L^2(m_2)}\\
 &&\qquad+\frac{\delta^{4/3}}2\langle \partial_xQ_\eps f,\partial_vf\rangle_{L^2(m_1)}+\frac{\delta^{4/3}}2\langle\partial_vQ_\eps f,\partial_xf\rangle_{L^2(m_1)}.
\end{eqnarray*}

To give an idea of the method, we only explain how to get rid of a few terms. For example, for the positive contribution of~\eqref{eq:two}, it holds
\begin{equation}\nonumber
 \delta^{3/2}k_1\|\partial_xf\|_{L^2(\R^2)}^2+\delta^{3/2}\int|\partial_xf||\partial_vf|m_2^2\leq\delta^{3/2}k_1\|\partial_xf\|_{L^2(\R^2)}^2+\delta^{7/4}\|\partial_xf\|_{L^2(m_2)}^2+\delta^{5/4}\,\|\partial_vf\|_{L^2(m_2)}^2,
\end{equation}
and for $\delta>0$ small enough these terms are annihilated by the quantities
$$
-\|\partial_vf\|^2_{L^2(m_2)}-\delta^{3/2}k_2\|\partial_xf\|^2_{L^2(M^{1/2}m_2)}-\frac {\delta^{4/3}b}2\|\partial_xf\|^2_{L^2( m_1)}
$$ present in the right hand side of~\eqref{eq:one},~\eqref{eq:two} and~\eqref{eq:four}.

In~\eqref{eq:three}, the only delicate contribution is
$$
\delta\,b\int|\partial_xf||\partial_vf|m_2^2\,\leq\, \frac{\delta^{5/3}\,b}2\|\partial_xf\|_{L^2(m_2)}^2+\frac{\delta^{1/3}b}2\|\partial_vf\|^2_{L^2(m_2)},
$$
but the right hand sides of~\eqref{eq:one} and~\eqref{eq:two} include
$$
-\|\partial_vf\|^2_{L^2(m_2)}-\delta^{3/2}k_2\|\partial_xf\|^2_{L^2(M^{1/2}m_2)},
$$
and once again for $\delta>0$ small the sum is nonpositive.

The positive part of~\eqref{eq:four} is controlled using that $\kappa_1<\kappa_2$. Indeed, in that situation
\begin{equation}\nonumber
 \delta^{4/3}k_0\int|\partial^2_{vv}f||\partial_xf|\,M^{3/2}\,m_1^2\leq \delta^{4/3-1/4}k_0\|\partial^2_{vv}f\|^2_{L^2(m_2)}+\delta^{4/3+1/4}k_0\|\partial_{x}f\|^2_{L^2(m_2)},
\end{equation}
replacing, if necessary, $k_0$ by a larger constant. If $\delta>0$ is small we get rid of these terms thanks to the presence of 
$$
-\delta^{3/2}k_2\|\partial_xf\|^2_{L^2(M^{1/2}m_2)}-\delta\,\|\partial^2_{vv}f\|^2_{L^2(m_2)},
$$ 
in~\eqref{eq:two} and~\eqref{eq:three}.

All remaining positive contributions can be handled in the same fashion leading to the conclusion that one can find $K_1,K_2>0$ such that
\begin{equation}\nonumber
 \langle Q_\eps f,f\rangle_{\HH^1}\leq K_1\|f\|_{L^2(\R^2)}^2-K_2\|f\|^2_{\HH^1}.
\end{equation}
\end{proof}

\begin{cor}\label{coro:H1m} Estimate \eqref{eq:H1m} holds.
\end{cor}

\begin{proof}
Nash's inequality in the 2-dimensional case reads: there exists a constant $C>0$, such that for any $f\in L^1(\R^2)\cap H^1(\R^2)$,
\begin{equation}\label{eq:Nash}
 \|f\|^2_{L^2(\R^2)}\,\,\leq\,\, C\|f\|_{L^1(\R^2)}\|D_{x,v} f\|_{L^2(\R^2)}\,\,\leq\,\,\frac{C}{2\delta'}\|f\|^2_{L^1(\R^2)}+\frac{\delta'}2\|D_{x,v} f\|^2_{L^2(\R^2)}.
\end{equation}
Coming back to the previous lemma, using the equivalence of the norms $\HH^1$ and $H^1(m_2)$, together with the fact that a solution $f_t$ to~\eqref{eq:FhNs} is a probability measure, we get that, 
\begin{equation}\nonumber
  \frac{d}{dt}\|f_t\|_{\HH^1}^2 \,=\, \langle Q_\eps [\JJJ_{f_t}]\,f_t,f_t\rangle_{\HH^1} \,\leq\, k_1-k_2\|f_t\|_{\HH^1}^2,
\end{equation}
for some $k_1,k_2>0$ constants. Finally, integrating in time, we get
$$
 \|f_t\|_{\HH^1}\leq \max(C_2,\|f_0\|_{\HH^1}),
$$
for some $C_2>0$ depending only on the parameters of the system and the initial condition.
\end{proof}

Let us notice that we can go a little further in the analysis of the regularity of the solutions of~\eqref{eq:FhNs}. Actually, we can expect that the norm $H^2_v(m)$ is also bounded. Indeed, there exists $k_0>0$ such that
\begin{eqnarray}\label{extra}
&& \langle \partial_{vv}^2 Q_\eps f,\partial_{vv}^2f\rangle_{L^2(m_2)} \\
&&\qquad=-\int|\partial^3_{vvv}f|\, m_2^2 + \frac12\int|\partial^2_{vv}f|^2\,\partial^2_{vv} m_2^2 \nonumber\\
&&\qquad\quad+2\int(\partial_vA)(\partial^2_{vv}f)(\partial^2_{xv}f)\,m_2^2+\frac12\int|\partial^2_{vv}f|^2\Big[\partial_xA-A\frac{\partial_x m_2^2}{ m_2^2}\Big]\, m_2^2 \nonumber\\
&&\qquad\quad+\int(\partial^3_{vvv}B) f(\partial^2_{vv}f)\, m_2^2+3\int(\partial^2_{vv}B)(\partial_vf)(\partial^2_{vv}f)\, m_2^2\nonumber\\
&&\qquad\quad+\frac12\int|\partial^2_{vv}f|^2\Big[5\,\partial_vB_\eps-B_\eps\frac{\partial_v m_2^2}{ m_2^2}\Big]\,m_2^2\nonumber 
\\
&&\qquad\leq k_0\,\Big[\int|\partial^2_{vv}f|^2+\int|\partial^2_{xv}f|^2\,m_2^2+\int|f|^2\,m_2^2+\int|\partial_{v}f|^2\,m_2^2\Big].
\nonumber
\end{eqnarray}
We can therefore state that

\begin{cor} 
Estimate~\eqref{eq:H2vm} holds.
\end{cor}

\begin{proof}
 The proof follows the same idea already introduced in the proof of Corollary~\ref{coro:H1m}. We consider the norm
 $$
\| f \|_{\HH^2_v}^2 := \| f \|_{\HH^1}^2 +  \delta^{2}\| \partial^2_{vv} f \|_{L^2(m_2)}^2,
$$
and notice that~\eqref{eq:one},~\eqref{eq:two},~\eqref{eq:three} together with~\eqref{extra} imply that
\begin{eqnarray*}
\frac{d}{dt}\| f_t \|_{\HH^2_v}^2 
&\leq&
 \frac{d}{dt}\|f_t\|^2_{\HH^1}\\
 &&\quad+2\, \delta^2\,k_0\,\Big[\int|\partial^2_{vv}f_t|^2m_2^2+\int|\partial^2_{xv}f_t|^2m_2^2+\int|f_t|^2m_2^2+\int|\partial_{v}f_t|^2m_2^2\Big]\\
 &\leq& k_1-k_2\| f_t \|_{\HH^2_v}^2 ,
\end{eqnarray*}
for some $k_1,k_2>0$ depending on some $\delta>0$ small and the parameters of the system. Inequality~\eqref{eq:H2vm} follows.
\end{proof}
\smallskip

\subsection{Entropy estimates and uniqueness of the solution}

Now we focus our attention on the problem of uniqueness of the solutions to~\eqref{eq:FhNs}. First, we prove that solutions remain in the space of functions with finite entropy. To that aim, for any positive function $f$, we define
$$
 I_v(f)\,\,:=\,\,\int_{\R^2}\frac{|\partial_vf(x,v)|^2}{f(x,v)}\,dxdv,
$$
which is understood as a {\it partial Fisher information}. When the previous quantity is not well defined we use the convention $I_v(f)=+\infty$. Notice that in any case $I_v(\cdot)\geq0$. Equipped with this definition we can state:

\begin{lem} For any $f_0 \in L^1(M) \cap L^1 \log L^1\cap \Pp(\R^2)$ we denote by $f_t$ the associated solution to the FhN statistical equation~\eqref{eq:FhNs} with initial condition $f_0$. It holds
\beqn\label{eq:entropy}
\sup_{t \in [0,T]} \HHH(f_t) + \int_0^t I_v(f_s) \, ds \le C(T),
\eeqn
where $C(T)$ depend on $f_0$ and the coefficients of the problem.
\end{lem}

\begin{proof}
It is well known that for functions with finite moments, the entropy can be bounded from below. Indeed, since
$$
 r_1\log r_1\,\geq\, -r_2+r_1\log r_2,\qquad\forall r_1\geq0,r_2>0,
$$
taking $r_1=f(x,v)$ and $r_2=e^{-M}$, it holds
$$
 0\,\,\geq\,\, f\,\log f \,\,\geq\,\, -e^{-M}- f\,M,
$$
implying that
$$
 \HHH(f_t)\geq -\int_{\R^2}e^{-M}-\int_{\R^2}f_t\,M\geq-2\pi e^{-1}-\max(C_0,\|f_0\|_{L^1(M)}).
 $$

On the other hand, for any solution of~\eqref{eq:FhNs} with initial datum $f_0$ there exists a positive constant $C$, depending on the parameters of the system, $\eps_0$ and $C_0'$, such that
\begin{eqnarray*}
 \frac{d}{dt}\HHH(f_t)&=&\int(1+\log(f_t))\,Q_\eps[\JJJ_{f_t}]\,f_t\\
 &=&-I_v(f_t)+\int\big(\partial_xA+\partial_vB_\eps(\JJJ_{f_t})\big)\,f_t\\
 &\leq&-I_v(f_t)+C\|f_t\|_{L^1(M)}.
\end{eqnarray*}
Let us fix $T>0$ and take any $t<T$, thanks to estimate~\eqref{eq:L1H}, we get that
$$
 \HHH(f_t)\leq-\int_0^tI_v(f_s)\,ds+ \HHH(f_0)+C\, T\,\max(C_0,\|f_0\|_{L^1(M)}).
$$
Since $\HHH$ is bounded by below, we get that $I_v(f_t)\in L^1([0,T])$. Moreover, taking the supremum on the last relationship, we get
$$
 \sup_{t\in[0,T]}\HHH(f_t)\,\,\leq\,\, \HHH(f_0)+C\, T\max(C_0,\|f_0\|_{L^1(M)}).
$$
\end{proof}

\begin{cor}\label{cor:Estimf-g} For any two initial data $f_0, g_0 \in L^1(M^2) \cap \,L^1 \log L^1\cap\, \Pp(\R^2)$ the associated solutions $f$ and $g$ to the FhN statistical equation \eqref{eq:FhNs}, satisfy
$$
\sup_{[0,T]} \| f_t - g_t \|_{L^1(M)} \le C(T) \, \| f_0 - g_0 \|_{L^1(M)},
$$
for some positive $C(T)$. In particular, equation~\eqref{eq:FhNs} with initial datum in $L^1(M^2) \cap L^1 \log L^1\cap \Pp(\R^2)$ has, at most, one solution.
\end{cor}

\noindent{\sl Proof of Corollary~\ref{cor:Estimf-g}. } We write 
$$
\partial_t (f_t-g_t) = Q_\eps[\JJJ(f_t)]\,(f_t-g_t) + \eps\,\JJJ(f_t-g_t)\, \partial_vg_t
$$
from which we deduce
\bean
{d \over dt} \int_{\R^2} |f_t-g_t| M
&\le& K_1 \int_{\R^2} |f_t-g_t| M + \eps\, |\JJJ(f_t-g_t)| \int_{\R^2} |\partial_v g_t|M
\\
&\leq&K_1 \int_{\R^2} |f_t-g_t| M + \eps\, |I(g_t)|^{1/2}\|g_t\|^{1/2}_{L^1(M^2)}\int_{\R^2}|f_t-g_t|M,
\eean
where $K_1$ is the constant introduced in the proof of Lemma~\ref{lemma:L1M}. Also, it is not hard to see that
$$
 \sup_{t\in[0,T]}\|g_t\|_{L^1(M^2)} \leq \|g_0\|_{L^1(M^2)}+2(K_1+1)T\max(C_0,\|g_0\|_{L^1(M)}).
$$
The rest of the proof is a direct application of the time integrability of $I_v(g_t)$ and Gronwall's lemma.
\qed
 \bigskip

 Let us finish this section by giving some insights of the proofs of the existence of solutions and stationary solutions to equation~\eqref{eq:FhNs} which are, however, classical.
 
 \begin{proof}[Proof of Theorem~\ref{th:E&U&B}]
 
 Let us consider an exponential weight $m$ and $\JJJ\in L^\infty(\R_+)$ such that 
 $$
 \sup_{t\geq0}|\JJJ|\leq C_0',
 $$ 
 where $C_0'$ is given by~\eqref{eq:mufbdd}. First, to avoid the non boundedness of the coefficients of the equation, let us fix $R>0$, and define a regular truncation function
\begin{equation}\label{eq:chiTruncation}
\chi_R(x,v) = \chi(x/R,v/R),\quad\chi \in \DD(\R^2),\quad{\bf 1}_{B(0,1)} \le \chi \le {\bf 1}_{B(0,2)}.
\end{equation}
Secondly, to avoid the intrinsic degenerate character of~\eqref{eq:FhNs}, we fix some $1>\sigma>0$, and define the bilinear form
\begin{eqnarray*}
a_\sigma(t;f,g) & := & \langle\partial_vf,\partial_v g\rangle_{L^2(m)}+\langle \partial_vf,g\,\chi_R\,m^{-2}\partial_{v} m^2\rangle_{L^2(m)}
\\
&&
\quad+\sigma\,\langle\partial_xf,\partial_xg\rangle_{L^2(m)}+\sigma\,\langle \partial_xf,g\,\chi_R\,m^{-2}\partial_{x} m^2\rangle_{L^2(m)}
\\
&&
\quad\quad-\frac12\langle f,g\,\chi_R\,[\partial_x A-A\,m^{-2}\partial_xm^2]\rangle_{L^2(m)}
\\
&&
\quad\quad\quad-\frac12\langle f,g\,\chi_R\,[\partial_v B_\eps(\JJJ_t)-B_\eps(\JJJ_t)\,m^{-2}\partial_vm^2]\rangle_{L^2(m)}.
\end{eqnarray*}
This bilinear form is obviously well defined, a.e. $t\geq0$, for any $f,g\in H^1(m)$. Moreover, $a_\sigma$ is continuous,
$$
|\,a_\sigma(t;f,g)|\,\,\leq\,\, C_R\|f\|_{H^1(m)}\|g\|_{H^1(m)},
$$
for some positive constant $C_R$, and coercive. Indeed, we have from~\eqref{eq:one}, that
$$
 a_\sigma(t;f,f)\,\,\geq\,\,\frac12\|\partial_{v} f\|^2_{L^2(m)}+\frac\sigma2\|\partial_{x} f\|^2_{L^2(m)}-k_1\|f\|^2_{L^2(m)},
$$
for some $k_1>0$ not depending on $t$, nor on $R$ and nor on $\sigma$. The J. L. Lions theorem~\cite[Theorem X.9]{brezis:83} implies that for any $f_0\in L^2(m)$ there exists a unique
$$
 f\in L^2((0,\infty); H^1(m))\cap C([0,\infty); L^2(m));\qquad \frac{d}{dt}f\in L^2((0,\infty);H^1(m)')
$$
such that $f(0)=f_0$ and
$$
\langle \frac{d}{dt}f, g\rangle_{L^2(m)}+a_\sigma(f(t),g)\,\,=\,\,0,\qquad \text{a.e. }t\geq0,\,\forall\, g\in H^1(m).
$$

We recall that $f_-:=\min(f,0)$ belongs to $H^1(m)$, therefore we can use it as a test function to find that
$$
f_0\geq0\qquad \Rightarrow\qquad f(t)\geq0,\qquad \text{a.e. }t\geq0.
$$
Let us now fix some $T>0$. Using $f$ itself as a test function, we get easily that
$$
 \|f_t\|_{L^2(m)}^2+\int_0^T\|\partial_v f_s\|_{L^2(m)}^2\,ds\,\,\leq\,\,e^{k_1T}\|f_0\|^2_{L^2(m)},
$$
therefore, one can take the limits $\sigma\rightarrow0$ and $R\rightarrow\infty$, to find that for any $\varphi\in C^1([0,T]; C_c^2(\R^2))$ 
$$
\int_{\R^2} \varphi_t f_t = \int_{\R^2}\varphi_0 f_0 +\int_0^t\int_{\R^2} \big[\partial_t\varphi_s+\partial^2_{vv}\varphi_s-A\,\partial_x\varphi_s-B_\eps(\JJJ_s)\partial_v\varphi_s\big] f_s\, ds,\quad 0< t<T,
$$
holds.
Taking a well chosen sequence $\varphi^n\rightarrow M^2$, we deduce that
$$
 \sup_{t\in[0,T]}\|f_t\|_{L^1(M^2)}\,\,\leq\,\, \max\big(C',\|f_0\|_{L^1(M^2)}\big),
$$
for some positive constant $C'$ that depends only on the parameters of the system. We also notice that, thanks to renormalisation concepts, we recover the inequality
$$
 \sup_{t\in[0,T]}\HHH(f_t)+\int_0^t I_v(f_s)\,ds\,\,\leq\,\, \HHH(f_0)+K_0 T\max(C_0,\|f_0\|_{L^1(M)}).
$$

Let us take now $f_0\in L^1(M^2)\cap L^1\log L^1\cap \Pp(\R^2)$, and a sequence $\{f_{n,0}\}\subset L^2(m)$ such that $f_{n,0}\rightarrow f_0$ in $L^1(M)$. Moreover, let us assume that there is a positive constant $C>0$ such that $\HHH(f_{n,0})\leq C$, for any $n\in\N$. From the previous analysis we get a family $\{f_{n}\}\in C((0,T);L^1(M))$ of functions related to the initial conditions $\{f_{n,0}\}$. Using the Dunford-Pettis criterium we can pass to the limit in $L^1(M)$ finding a solution to the linear problem
\begin{equation}\label{eq:lineal}
 \partial_t f\,\,=\,\,\partial_x(Af)+\partial_v(B_\eps(\JJJ_t)f)+\partial^2_{vv}f.
\end{equation}
that depends continuously to the initial datum (in the sense defined in Theorem~\ref{th:E&U&B}). Moreover, from Corollary~\ref{cor:Estimf-g} we get that this solution is necessarily unique.

Finally,  we use again the ideas of Corollary~\ref{cor:Estimf-g} to find a solution to the NL equation~\eqref{eq:FhNs}. Indeed, it suffices to notice that the mapping
\[
\left\{\begin{array}{rcl}
 L^\infty([0,T]) &\longrightarrow& C([0,T]; L^1(M^2))
 \\
 \JJJ &\longmapsto& f,
 \end{array}
 \right.
\]
with $f$ solution of~\eqref{eq:lineal} for $\JJJ$ given, is Lipschitz and contracting when $T>0$ is small enough.
 \end{proof}


Existence of stationary solutions will be shown as a result of an abstract version of the Brouwer fixed point theorem (a variant of~\cite[Theorem 1.2]{EMRR} and \cite{MR2533928}):

\begin{theo}\label{th:AbsStat} Consider $\ZZ$ a convex and compact subset of a Banach space $X$ and $S(t)$ a continuous semigroup on $\ZZ$. Let us assume that $\ZZ$ is invariant under the action of $S(t)$ (that is $S(t) z \in \ZZ$ for any $z \in \ZZ$ and $t \ge 0$). Then, there exists $z_0 \in \ZZ$ which is stationary under the action of $S(t)$, i.e, $S(t) z_0 = z_0$ for any $t \ge 0$.
\end{theo}

We present the argument briefly in this section. Our aim is to find a fixed point for the nonlinear semigroup $S_{Q_\eps}(t)$ related to equation~\eqref{eq:FhNs}. At this point we do not have any hint on the number of functions solving
$$
 Q_\eps[\JJJ_F]\,F=0,
$$
and the nonlinearity could lead to the presence of more than one. However, in the disconnected regime $\eps=0$ the nonlinearity disappears, and the multiplicity problem is no longer present. 

\begin{proof}[Proof of existence of stationary solutions to~\eqref{eq:StatExist}]

Let us fix $m$ an exponential weight and define for any $t\geq0$
$$
S(t): X \,\rightarrow\, X\quad\text{with}\quad X\,=\, H^2_v(m)\cap L^1\log L^1\cap\Pp(\R^2),
$$
such that $S(t)f_0$ is the solution to~\eqref{eq:FhNs} given by Theorem~\ref{th:E&U&B} associated to the initial condition $f_0$. Estimates~\eqref{eq:H2vm} and~\eqref{eq:entropy} imply that $S(t)$ is well defined. Moreover, the continuity of $S$ in the Banach space $L^1(\R^2)$ is direct from the definition of weak solutions, in particular,
$$
S(t)f_0\in C([0,\infty);L^1(\R^2)),
$$ 
with the topology of compact subsets in time.

Finally, defining 
$$
\ZZ:=\ZZ(\eps)=\{ f\in X\mbox{ such that~\eqref{eq:L1H} and~\eqref{eq:H2vm} hold}\}\subset L^1(\R^2),
$$
which is invariant under $S_t$ for any $t\geq0$ and convex. Moreover, the compactness of the inclusion
$
 \ZZ\subset H^1(m)\hookrightarrow L^1(\R^2)
$ 
allows us to apply Theorem~\ref{th:AbsStat} and find the existence of a fixed point for $S(t)$ and by consequence a stationary solution to~\eqref{eq:FhNs}.

It is worth emphasising that the above proof show yet that the map $\eps \mapsto G_\eps$ is locally bounded in $[0,\infty)$, i.e., if $\eps_0>0$ is fixed, then $$
G_\eps\in \ZZ(\eps_0)\quad\text{ for any }\quad\eps\in(0,\eps_0).
$$
 \end{proof}



\section{The linearized equation} 
\label{sec:UniqStatSol}
\setcounter{equation}{0}
\setcounter{theo}{0}

The aim of the present section is to undercover the properties of the linearized operator associated to $Q_\eps$ in the small connectivity case using what we call a \textit{splitting method}. 
To illustrate the ideas we use, let us assume that an operator $\Lambda$ on a Banach space $X$ can be written as
$$
 \Lambda\,=\,\AA+\BB,
$$
where $\AA$ is \emph{much more regular than $\BB$}, and $\BB$ has some \emph{dissipative property}. If $\BB$ has a good localisation of its spectrum, under some reasonable hypotheses on $\AA$, we expect $\Sigma(\Lambda)$ to be close to $\Sigma(\BB)$.

This is nothing but the Weyl's abstract theorem (and/or the generalisation of the Krein-Rutman theorem) from Mischler and Scher~\cite{MS*}, that we recall here:
\begin{theo}\label{th:KRG}
 We consider a semigroup generator $\Lambda$ on a ``Banach lattice of functions'' $X$, and we assume that
 \begin{enumerate}
  \item there exists some $\alpha^*\in\R$ and two operators $\AA,\BB\in\CCC(X)$, such that $\Lambda=\AA+\BB$ and
  \begin{enumerate}
  \item for any $\alpha>\alpha^*,\ell\geq0$, there exists a constant $C_{\alpha,\ell}>0$ such that 
  $$\forall\,t\geq0,\qquad \|S_\BB*(\AA S_\BB)^{(*\ell)}(t)\|_{\BBB(X)}\leq C_{\alpha,\ell}\,e^{\alpha t}.
  $$
  \item$\AA$ is bounded, and there exists an integer $n\geq1$ such that for any $\alpha>\alpha^*$, there exists a constant $C_{\alpha,n}>0$ such that
  $$
  \forall\,t\geq0,\qquad \|(\AA S_\BB)^{(*n)}(t)\|_{\BBB(X,Y)}\leq C_{\alpha,n} e^{\alpha t},
  $$ 
  with $Y\subset D(\Lambda)$ and $Y\subset X$ with compact embedding;
  \end{enumerate}
  \item for $\Lambda^*$ the dual operator of $\Lambda$ defined in $X'$, there exists $\beta>\alpha^*$ and $\psi\in D(\Lambda^*)\cap X'_+\setminus\{0\}$ such that $$\Lambda^*\psi\geq \beta\psi;$$
  \item $S_\Lambda$ satisfies Kato's inequalities, i.e,
  $$
   \forall\, f\in D(\Lambda),\quad \Lambda \theta(f)\geq \theta'(f)\,\Lambda f,
  $$
  holds for $\theta(s)=|s|$ or $\theta(s)=s_+$.
  \item $-\Lambda$ satisfies a strong maximum principle: for any given $f$ and $\gamma\in\R$, there holds,
  $$
   |f|\in D(\Lambda)\setminus\{0\}\text{ and }(-\Lambda+\gamma)|f|\geq0\quad\text{ imply }\quad f>0\text{ or }f<0.
  $$
 \end{enumerate}
Defining 
$$
 \lambda\,:=\,s(\Lambda)=\sup\big\{\re(\xi)\,\,:\,\,\xi\in\Sigma(\Lambda)\big\},
$$
 there exists $0<f_\infty\in D(\Lambda)$ and $0<\phi\in D(\Lambda^*)$ such that 
$$
 \Lambda f_\infty=\lambda\, f_\infty,\qquad\Lambda^*\phi=\lambda\,\phi.
$$ 
Moreover, there is some $\bar\alpha\in(\alpha^*,\lambda)$ and $C>0$ such that for any $f_0\in X$ $$\|S_\Lambda(t)f_0-e^{\lambda t}\langle f_0,\phi\rangle f_\infty\|_X\leq C e^{\bar\alpha t}\|f_0-\langle f_0,\phi\rangle f_\infty\|_X.$$
\end{theo}
\bigskip


 From Theorem~\ref{th:StatExist} we know that for any value of $\eps$ there exists at least one $G_\eps$ non zero stationary solution of the FhN kinetic equation~\eqref{eq:FhNs}. The linearized equation, on the variation $h\,:=\, f-G_\eps$, 
induces the linearized operator
$$
\LLL_\eps h =   Q_\eps(\JJJ(G_\eps))h + \eps \,  \JJJ(h) \partial_vG_\eps.
$$
Moreover, let us recall that in Section~\ref{sec:estim} we proved that
$$
 \langle Q_\eps[\JJJ(G_\eps)]\,f,f\rangle_{L^2(m)}\,\leq\, K_1\|f\|_{L^2(\R^2)}-K_2\|f\|_{L^2(m)},
$$
if we could make $K_1=0$, then the operator $Q_\eps$ together with $\LLL_\eps$ would be dissipative. Since it is not the case, let us fix a constant $N>0$ and define
\begin{equation}\label{eq:splitAB}
 \BB_\eps\,:=\,\LLL_\eps - \AA, \quad\text{where}\quad \AA = N\, \chi_R(x,v);
\end{equation}
with $\chi_R$ given by~\eqref{eq:chiTruncation}. We remark that $\AA \in \BBB(H^2_v(m))$, and that $\AA f$ vanishes outside a ball of radius $2R$ for any $f\in H^2_v(m)$.
\smallskip

\subsection{Properties of $\AA$ and $\BB_\eps$} We now precise the dissipative properties of $\LLL_\eps$. In particular, we present two lemmas dealing with the hypodissipativity and regularisation properties of the sppliting $\AA$ and $\BB_\eps$. We use some ideas developed in~\cite{MMcmp,GMM} and~\cite{MM*}.

\begin{lem}\label{lemma:Bdiss}
 For any exponential weight $m$, there exist some constants $N,R>0$ such that $(\BB_\eps+1)$ is hypodissipative in $H^2_v(m)$.
\end{lem}

\begin{proof}

From the characterisation of hypodissipativity given in Section~\ref{sec:summary}, it suffices to show that there exists a constant $C>0$ such that
$$
 \|S_{\BB_\eps}(t)\|_{\BBB(H^2_v(m))}\leq C\,e^{-t},\quad t\geq0,
$$
or simply, to show that for any $h\in H^2_v(m)$, it holds
\begin{equation}\label{eq:H2vbarNorm}
 \langle\BB_\eps h,h\rangle_{\bar H^2_v(m)} \leq -\|h\|_{\bar H^2_v(m)}^2,
\end{equation}
for some norm $\|\cdot\|_{\bar H^2_v(m)}$ equivalent to the usual norm $\|\cdot\|_{H^2_v(m)}$.

Let us recall that the operator $\BB_\eps$ writes
$$
 \BB_\eps \,=\, \LLL_\eps-\AA \,=\, (Q_\eps[\JJJ_{G_\eps}]-N\chi_R)h+\eps\,\JJJ(h)\,\partial_v G_\eps,
$$
and since $\JJJ_{G_\eps}\in\R$ is a real constant, we can use all a priori estimates on $Q_\eps$ directly. As usual, when no confusion is possible, we drop the dependence on $\JJJ_\eps$. Three steps complete the proof:\\

\noindent{\sl Step 1. Dissipativity in $L^2(m).$}  Let us notice that for any $h\in L^2(m)$ we have
\begin{equation}\nonumber
 |\JJJ(h)|\leq C\|h\|_{L^2(m)},
\end{equation}
for some constant $C>0$. It follows that
\begin{equation}\nonumber
 \JJJ(h)\int_{\R^2}(\partial_vG_\eps)\,h\,m^2\leq|\JJJ(h)|\|\partial_vG_\eps\|_{L^2(m)}\|h\|_{L^2(m)}\leq C\,\|\partial_vG_\eps\|_{L^2(m)}\int_{\R^2} h^2m^2.
\end{equation}

Thus, coming back to~\eqref{eq:preOne}, we find that for $N$ and $R$ large enough one can assume $k_1=-1$, getting
\begin{eqnarray}\label{eq:dissipativeL2}
 \langle\BB_\eps h,h\rangle_{L^2(m)}&\leq& -\|h\|^2_{L^2(m)}-k_2\|h\|^2_{L^2(M^{1/2}m)}-\|\partial_vh\|^2_{L^2(m)},
\end{eqnarray}
 as a consequence, $(\BB_\eps+1)$ is dissipative in $L^2(m)$. 
 \bigskip

\noindent{\sl Step 2. Bounds on the derivatives of $\BB_\eps$.}
 For the $x$-derivative we see that there exists some constant $C'$ depending on $\chi_R$ and its derivatives, such that
 $$
-N\langle\partial_x(\chi_Rh),\partial_xh\rangle_{L^2(m)}\leq C'\| h\|_{L^2(m)}^2- N\|(\partial_xh)\sqrt{\chi_R}\|_{L^2(m)}^2.
$$
On the other hand, thanks to Young's inequality, we get
\begin{eqnarray*}
 \JJJ(h)\int_{\R^2}(\partial_{xv}^2G_\eps)(\partial_xh)\,m^2&=&-\JJJ(h)\int_{\R^2}\partial_xG_\eps\big[\partial^2_{vx}h+2\kappa v\,\partial_xh\big]\,m^2
 \\
 &\leq& \JJJ(h)^2\,\|\partial_vG_\eps\|_{L^2(m)}^2+\frac12\|\partial^2_{xv}h\|_{L^2(m)}^2+\|\sqrt{2}\kappa v\,\partial_xh\|_{L^2(m)}^2.
 \end{eqnarray*}
These two inequalities, together with~\eqref{eq:two}, imply that  for $N$ and $R$ large enough
 \begin{equation}\nonumber
  \langle\partial_x(\BB_\eps h),\partial_xh\rangle_{L^2(m)}\leq-\|\partial_xh\|^2_{L^2(m)}-\frac12\|\partial^2_{xv}h\|^2_{L^2(m)}+C'\|h\|^2_{L^2(m)}+\int_{\R^2}|\partial_xh|\,|\partial_vh|m^2.
 \end{equation}
 \smallskip

 Proceeding similarly with the $v$-derivative we get
\begin{eqnarray*}
 \JJJ(h)\int_{\R^2}(\partial^2_{vv}G_\eps)(\partial_vh)\,m^2&=&|\JJJ(h)|\|\partial^2_{vv}G_\eps\|_{L^2(m)}\|\partial_vh\|_{L^2(m)}
 \\
 &\leq&\frac{1}2\,\|\partial^2_{vv}G_\eps\|_{L^2(m)}(C^2\|h\|^2_{L^2(m)}+\|\partial_vh\|_{L^2(m)}^2),
 \end{eqnarray*}
then, coming back to~\eqref{eq:three}, we find $N,R>0$ such that
\begin{equation}\nonumber
  \langle\partial_v(\BB_\eps h),\partial_vh\rangle_{L^2(m)}\leq-\|\partial_vh\|^2_{L^2(m)}-\|\partial^2_{vv}h\|^2_{L^2(m)}+C'\,\|h\|^2_{L^2(m)}+\int_{\R^2}|\partial_xh|\,|\partial_vh|\,m^2.
 \end{equation}
 \smallskip
 
 Finally, for the second $v$-derivative we find $C'$ such that
 \begin{eqnarray*}
-N\langle\partial^2_{vv}(\chi_Rh),\partial^2_{vv}h\rangle_{L^2(m)}&\leq&-N\int_{\R^2}\chi_R(\partial^2_{vv}h)^2m^2+C'\,\int_{\R^2} (\partial_vh)^2m^2+C'\,\int_{\R^2} h\,|\partial_{vv}^2h|\,m^2,
\end{eqnarray*}
and for any $\epsilon>0$
\begin{eqnarray*}
\JJJ(h)\int_{\R^2}(\partial^3_{vvv}G_\eps)(\partial^2_{vv}h)\,m^2&\leq&\frac{\JJJ(h)^2}{2\epsilon}+\epsilon\,\big(\|\partial^2_{vv}G_\eps\|_{L^2(m)}^2\|\partial^3_{vvv}h\|_{L^2(m)}^2+
\\ 
&&\qquad+\|\partial^2_{vv}G_\eps\|_{L^2(m)}^2\|2\kappa v\,(\partial^2_{vv}h)\|_{L^2(m)}^2\big).
\end{eqnarray*}
 If $\epsilon>0$ is small and $N,R$ large enough, we obtain as an application of~\eqref{extra}, that there is a constant $C'>0$ such that
  \begin{equation}\nonumber
  \langle\partial^2_{vv}(\BB_\eps h),\partial^2_{vv}h\rangle_{L^2(m)}
  \,\leq\,-\|\partial^2_{vv}h\|^2_{L^2(m)}+C'\big[\|h\|^2_{L^2(m)}+\|\partial_vh\|^2_{L^2(m)}+\|\partial^2_{xv}h\|^2_{L^2(m)}+\|\partial^2_{vv}h\|^2_{L^2(m)}\big].
 \end{equation}
 \bigskip
 
 \noindent{\sl Step 3. Equivalent norm and conclusion.} Let $\delta>0$ and $h_1,h_2\in H^2_v(m)$, we can define the bilinear product
 \begin{equation}\nonumber
  \langle h_1,h_2\rangle_{\bar H^2_v(m)}:=\langle h_1,h_2\rangle_{L^2(m)}+\delta\langle\partial_x h_1,\partial_xh_2\rangle_{L^2(m)}+\delta\langle\partial_v h_1,\partial_vh_2\rangle_{L^2(m)}+\delta^2\langle\partial^2_{vv} h_1,\partial^2_{vv}h_2\rangle_{L^2(m)}.
 \end{equation}
and the relative norm 
$$
 \|h\|_{\bar H^2_v(m)}^2:=\|h\|_{L^2(m)}^2+\delta\,\|D_{x,v} h\|_{L^2(m)}^2+\delta^2\,\|\partial^2_{vv}h\|_{L^2(m)}^2.
$$
Choosing $\delta>0$ small enough we conclude that for any $\alpha\in(0,1]$ one find $\delta_\alpha$ such that
$$
 \langle \BB_\eps h,h\rangle_{\bar H^2_v(m)} \leq -\alpha\,\|h\|^2_{\bar H^2_v(m)}.
$$
Since the norm related to $\bar H^2_v(m)$ is equivalent to the usual norm in $H^2_v(m)$, we can conclude that $(\BB_\eps+1)$ is hypodissipative in $H^2_v(m)$.
\end{proof}\bigskip


\begin{lem}\label{lemma:RegSB2}
 There are positive constants $N,R$ large enough and some $C_{\BB_\eps}>0$, such that the semigroup $S_{\BB_\eps}$ satisfies
 $$
 \|S_{\BB_\eps}(t) h\|_{H^2_v(m_1)}\leq C_{\BB_\eps} t^{-9/2}\| h\|_{L^2(m_2)},\quad \forall\,t\in(0,1].
 $$ 
As a consequence, for any $\alpha>-1$, and any exponential weight $m$, there exists $n\geq1$ and $C_{n,\eps}$ such that of any $t>0$ it holds
\begin{equation}\label{eq:consequence}
 \|(\AA S_{\BB_\eps})^{(\ast n)}(t)h\|_{H^2_v(m)}\leq C_{n,\eps}\, e^{\alpha t}\|h\|_{L^2(m)}.
\end{equation}
\end{lem}\smallskip

\begin{proof} We split the proof in three steps, in the first one we refine the previous estimates on the norm of the semigroup associated to the operator $\BB_\eps$, in the second one we use Hormander-H\'erau technique (see e.g.~\cite{MR2294477}) to get the first inequality, and finally we prove~\eqref{eq:consequence}.\\

\noindent{\sl Step 1. Sharper estimates on $\BB_\eps$.}
We denote for $K>0$ a generic constant. From the proof of the previous Lemma, we know that there are $N,R$ large enough such that for any $h\in D(\BB_\eps)$, it holds
\begin{eqnarray*}
\langle \BB_\eps h,h\rangle_{L^2(m_2)} &\le& -K\|h\|^2_{L^2(m_2)}-\|\partial_vh\|^2_{L^2(m_2)}
\\
\langle \partial_x\BB_\eps h,\partial_xh\rangle_{L^2(m_1)} &\le& -\frac12\|\partial_xh\|^2_{L^2(m_1)}-\frac12\|\partial^2_{xv}h\|^2_{L^2(m_1)}+K\|h\|^2_{L^2(m_1)}+\frac1{2\delta t}\|\partial_vh\|^2_{L^2(m_1)}
\\
\langle \partial_v\BB_\eps h,\partial_vh\rangle_{L^2(m_1)} &\le& -\|\partial^2_{vv}h\|^2_{L^2(m_1)}+K\|h\|^2_{L^2(m_1)}+\frac1{\delta t}\|\partial_vh\|^2_{L^2(m_1)}+\delta t\|\partial_xh\|^2_{L^2(m_1)}
\\
\langle\partial^2_{vv}\BB_\eps h,\partial^2_{vv}h\rangle_{L^2(m_1)} &\le& K\|h\|^2_{L^2(m_1)}+K\|\partial_vh\|^2_{L^2(m_1)}+\frac1{2t\delta}\|\partial^2_{xv}h\|^2_{L^2(m_1)}.
\end{eqnarray*}
We also notice for any $\delta,t\in(0,1)$ it holds
\begin{multline}\nonumber
\langle\partial_x(Q_\eps-\AA) h, \partial_v h \rangle_{L^2(m_1)}+\langle\partial_v (Q_\eps-\AA) h,\partial_xh \rangle_{L^2(m_1)}
\leq
-\frac{b}2\|\partial_xh\|^2_{L^2(m_1)}\\+K\|h\|^2_{L^2(m_2)}+\frac{K}{t\delta}\|\partial_vh\|_{L^2(m_2)}^2
+\frac{K}{t\delta^{1/10}}\|\partial_{vv}^2h\|_{L^2(m_1)}^2+Kt\delta^{1/10}\|\partial^2_{xv}h\|_{L^2(m_1)}^2,
\end{multline}
\begin{multline}\nonumber
\langle\JJJ(h)\partial^2_{xv}G_\eps,\partial_v h\rangle_{L^2(m_1)}+\langle\JJJ(h)\partial^2_{vv}G_\eps,\partial_x h\rangle_{L^2(m_1)}
\leq 
\frac{\|\partial_xG_\eps\|^2_{L^2(m_1)}}{2}\big[2\JJJ(h)^2
\\
+\|\partial^2_{vv}h\|^2_{L^2(m_1)}+\|\partial_{v}h\|^2_{L^2(m_2)}\big]
+\frac{\|\partial^2_{vv}G_\eps\|_{L^2(m_1)}}2\Big[\frac{\JJJ(h)^2}{t\delta}+t\delta\|\partial_{x}h\|^2_{L^2(m_1)}\Big],
\end{multline}
yielding to
\begin{eqnarray*}
\langle\partial_x\BB_\eps h, \partial_v h \rangle_{L^2(m)}+\langle\partial_v \BB_\eps h,\partial_xh \rangle_{L^2(m)}
&\leq&
-\frac{b}4\|\partial_x h\|^2_{L^2(m_1)}+\frac K{t\delta}\|h\|^2_{L^2(m_2)}+\frac K{t\delta}\|\partial_vh\|^2_{L^2(m_2)}
\\
&&\quad+\frac{K}{t\delta^{1/10}}\|\partial_{vv}^2h\|_{L^2(m_1)}^2+Kt\delta^{1/10}\|\partial^2_{xv}h\|_{L^2(m_1)}^2.
\end{eqnarray*}

\noindent{\sl Step 2. Hormander-H\'erau technique.}
For a given $h\in H^2_v(m_1)\cap L^2(m_2)$ we denote $h_t:=S_{\BB_\eps}(t) h$, and define $\FF$ by
$$
\FF(h,t) :=  \| h \|_{L^2(m_2)}^2 + c_1t^3 \| \partial_x h \|_{L^2(m_1)}^2 +  c_2t  \| \partial_v h \|_{L^2(m_1)}^2
+  c_3t^2 \langle\partial_x h,\partial_v h\rangle_{L^2(m_1)} +c_4 t^4\|\partial^2_{vv}h\|^2_{L^2(m_1)},
$$
which, for well chosen parameters, is decreasing. Indeed, thanks to the inequalities found in the first step, we have
\begin{eqnarray*}
 \frac{d}{dt}\FF(t,h_t) &\leq&
\sum_{i=1}^5 \TT_i,
 \end{eqnarray*}
 with
 \begin{eqnarray*}
 \TT_1 &=& K\int_{\R^2}\big[-2m_2^2+2(c_1t^3+c_2 t+c_4t^4)m_1^2+\frac{c_3 t}{\delta}m_2^2\,\big]h_t^2, \\
 \TT_2 &=& \int_{R^2}\big[(3c_1+2c_2\delta-\frac b4c_3+2c_3\delta)t^2-c_1t^3\big](\partial_x h_t)^2m_1^2, \\
 \TT_3 &=& \int_{\R^2}\big[-2m_2^2+c_2m_1^2+\frac{2c_2}\delta m_1^2+\frac{2c_3}{\delta}m_1^2+\frac{c_1t^2}{\delta} m_1^2+2c_4t^4Km_1^2+\frac{c_3tK}\delta m_2^2\big](\partial_v h_t)^2, \\
 \TT_4 &=&  \int_{\R^2}t^3\big[-c_1+\frac{c_4}{\delta}+c_3K\delta^{1/10}\big](\partial^2_{xv} h_t)^2m_1^2, \\
 \TT_5 &=& \int_{\R^2}\big[-2c_2t+\frac{c_3tK}{\delta^{1/10}}+4c_4t^3\big](\partial^2_{vv} h_t)^2m_1^2.
\end{eqnarray*}
 Choosing
 \begin{equation}\nonumber
  c_1=\delta^2,\quad c_2=\delta^{4/3}\quad c_3=\delta^{3/2}\quad\text{and}\quad c_4=\delta^4,
 \end{equation}
we get that for $\delta\in(0,1]$ small enough, it holds
 \begin{eqnarray*}
 \frac{d}{dt}\FF(t,h_t) &\leq& 0.
 \end{eqnarray*}
  for any $t\in(0,1]$. Since $0<c_4\leq c_1\leq c_3\leq c_2$ and $c_1c_2\geq c_3^2$, we finally get that
 \begin{equation}\nonumber
  c_4\, t^{9/2}\Big(\|\partial_{x,v}h_t\|^2_{L^2(m_1)}+\|\partial^2_{vv} h_t\|^2_{L^2(m_1)}\Big)\leq\FF(t,h_t)\leq F(0,h_0)=\|h_0\|^2_{L^2(m_2)}.
 \end{equation}

\noindent{\sl Step 3. Proof of inequality~\eqref{eq:consequence}.} From the definition of $\AA$ we notice that
 $$
 \|\AA\,S_{\BB_\eps}(t) h\|_{H^2_v(m)}\leq C' t^{-9/2}e^{-t}\| h\|_{L^2(m)},\quad \forall\,t\in(0,1],
 $$ 
for some constant $C'$. It is important to remark that since $\AA$ lies in a compact, we do not need anymore two different weights $m_1$ and $m_2$. Therefore, we apply Proposition~\ref{prop:semigroupGeneral} with $X=L^2(m)$, $Y=H^2_v(m)$, $\Theta=9/2$ and $\alpha^*=-1$ to get~\eqref{eq:consequence}.
\end{proof}

\subsection{Spectral analysis on the linear operator in the disconnected case}
\label{subsec:Scher}

We consider in this section the disconnected case $\eps = 0$. 
The corresponding FhN kinetic equation is linear and writes
\bean
&&\partial_t g \,=\, \partial_x(A g) + \partial_v(B_0 \,g) + \partial_{vv}^2 g
\\
&& B_0 = v\,(v-\lambda)\,(v-1) + x,
\eean
Theorem~\ref{th:StatExist} states that there exists at least one function $G_0 \in \Pp \cap H^2_v(m)$ which is a solution to the associated (linear) stationary problem 
$$
\LLL_0  G_0=  \partial_x ( A  G_0 ) + \partial_v ( B_0 G_0 ) +  \partial^2_{vv} G_0 = 0.
$$
Since the operator now enjoys a \emph{positive structure} (it generates  a positive semigroup $S_{\LLL_0}$), we can perform a more accurate analysis. Indeed, we can apply the the abstract Krein-Rutman theorem~\ref{th:KRG} previously stated.
%
%
%

\begin{proof}[Proof of the stability around $\eps=0$ in Theorem~\ref{th:StatExist}]
Let us assume for a first moment that hypotheses of the abstract Theorem~\ref{th:KRG} hold for $\LLL_0$ with $\alpha^*=-1$. We easily remark that
$$
 \lambda=0,\quad f_\infty=G_0\quad \phi=1,
$$
therefore, there exists $\bar\alpha\in (-1,0)$ such that
$$
\Sigma(\LLL_0) \cap \Delta_{\bar\alpha} = \{ 0 \},
$$
and
$$
\forall \, f_0 \in L^2(m), \,\, \forall \, t \ge 0 \qquad
\| S_{\LLL_0}(t) f_0 - \langle f_0 \rangle G_0 \|_{L^2(m)} \le 
C \, e^{\bar\alpha t} \|f_0 - \langle f_0 \rangle G_0 \|_{L^2(m)}.
$$

Now, for $\eps>0$, we consider $G_\eps$ such that
$$
 Q_\eps[\JJJ_{G_\eps}]\,G_\eps=0,
$$
then, it holds
 \begin{equation}\nonumber
  \frac{\partial}{\partial t}(G_\eps-G_0)+\LLL_0(G_\eps-G_0)=h,\qquad h=\eps \, \partial_v((v-\JJJ(G_\eps))G_\eps),
 \end{equation}
and, thanks to Duhamel's formula, we get that
$$
\|G_\eps-G_0\|_{L^2(m)}\leq \|S_{\LLL_0}(t)(G_\eps-G_0)\|_{L^2(m)}+\int_0^t \|S_{\LLL_0}(t-s) h\|_{L^2(m)}\,ds.
$$
But $G_\eps-G_0$ and $h$ have zero mean, then
$$
\|G_\eps-G_0\|_{L^2(m)}\leq C\|G_\eps-G_0\|_{L^2(m)}e^{\bar\alpha t}+\eps\, \frac{C}{|\bar\alpha|}\|G_\eps\|_{H^1_v(M^{1/2}m)}(1-e^{\bar\alpha t}).
$$
Letting $t\rightarrow\infty$ we conclude that there exists $C_{\bar\alpha}>0$ such that
$$
\|G_\eps-G_0\|_{L^2(m)}\leq \eps\, C_{\bar\alpha}\|G_\eps\|_{H^1_v(M^{1/2}m)}.
$$

Finally, thanks to Corollary~\ref{coro:H1m}, we have
$$
  0\,\,=\,\,\langle Q_\eps [\JJJ_{G_\eps}]G_\eps,G_\eps\rangle_{\HH^1}\,\,\leq\,\, K_1-K_2\|G_\eps\|_{\HH^1}^2\,\,\leq\,\,K_1-c_\delta\,K_2\|G_\eps\|_{H^1(m_2)}^2,
$$
for any exponential weight $m_2$. If $\kappa_2>\kappa$, we have then
$$
\|G_\eps\|_{H^1_v(M^{1/2}m)}^2\,\,\leq\,\, C_{\kappa,\kappa_2}\|G_\eps\|_{H^1(m_2)}^2\,\,\leq\,\, C_{\kappa,\kappa_2}K_1/c_\delta K_2,
$$
and in the small connectivity regime $\eps\in(0,\eps_0)$, constants $K_{1}$ and $K_2$ do not depend on $\eps$. Defining $\eta(\eps)= \eps\, C_{\bar\alpha} C_{\kappa,\kappa_2}K_{1}/c_\delta K_2$ we get the stability part of Theorem~\ref{th:StatExist}.\\

It only remains to verify that the requirement of Theorem~\ref{th:KRG} are fulfilled for $\LLL_0$ in the Banach lattice $X=L^2(m)$. \\
\begin{enumerate}
\item
\begin{itemize}
\item[(a)] the splitting~\eqref{eq:splitAB} has the nice structure. Indeed, the Lemma~\ref{lemma:Bdiss} implies that $\BB_0+1$ is hypodissipative in $L^2(m)$, therefore
$$
 \|S_{\BB_0}(t)\|_{\BBB(L^2(m))}\leq C e^{-t},\quad \forall\,t\geq0,
$$
i.e., it suffices to take $\alpha^*=-1$.
\item[(b)] if $Y=H^2_v(m)$ and $X=L^2(m)$, the desired inequality is consequence of Lemma~\ref{lemma:RegSB2}.
\end{itemize}

\item The requirement is obtained for $\beta=0$ and $\psi=1$. Indeed, in that case
$$
\LLL_0^*\psi \,=\, Q_0^* 1 \,=\, 0\,\geq\, \beta\psi.
$$
\item A side consequence of~\eqref{eq:one} is the positivity of the semigroup: $$f_0\geq0\qquad \Rightarrow\qquad S_{\LLL_0}f_0(t)\geq0,\quad\forall\, t\geq0.$$
Moreover, using that $L^2(m)$ is also a Hilbert space, we deduce the Kato's inequalities.
\item The strict positivity (or strong maximum principle) is a straightforward consequence of Theorem~\ref{th:FPtheo} in Appendix~\ref{app:MaximumPrinciple}.
\end{enumerate}
 \end{proof}

Let us finish this section by summarizing the properties of the spectrum of $\LLL_0$ in the Banach space $L^2(m)$ and by a useful result on the regularisation properties of $\RR_{\LLL_0}(z)$.
 \begin{prop} \label{prop:L0Ext}\ 
  \begin{itemize}
   \item[(i)] There exists $\bar\alpha<0$ such that the spectrum $\Sigma(\LLL_0)$ of $\LLL_0$ in $L^2(m)$ writes 
   $$
   \Sigma(\LLL_0)\cap \Delta_{\bar\alpha}\,=\,\{0\},
   $$
   and $0$ is simple.
   \item[(ii)] For any $\alpha>\bar\alpha$, there exists a constant $C_{H^1_v}>0$ depending on $(\alpha-\bar\alpha)$, such that
   $$
    \|\RR_{\LLL_0}(z)\|_{\BBB(L^2(m),H^1_v(m))}\leq C_{H^1_v}(1+|z|^{-1}),\qquad  \forall\,z\in\C\setminus\{0\},\text{Re}(z)>\alpha.
   $$
  \end{itemize}
 \end{prop}
\begin{proof}
It only remains to prove (ii). Let us consider $z\in \Delta_\alpha\setminus\{0\}$, and take $f,g\in L^2(m)$ such that
$$
 (\LLL_0-z)f\,\,=\,\, g.
$$
Thanks to Lemma~\ref{lemma:Bdiss} and the definition of $\AA$, we get
$$
 (\text{Re}(z)-\bar\alpha)\|f\|_{L^2(m)}^2+\|\partial_v f\|^2_{L^2(m)} \,\,\leq\,\, \| g\|_{L^2(m)}\,\|f\|_{L^2(m)} + N\,\|f\|^2_{L^2(m)}.
$$
Moreover, (i) tells us that $0$ is an isolated simple eigenvalue for $\LLL_0$ in $L^2(m)$, then $\RR_{\LLL_0}(z)$ writes as the Laurent series (see for example~\cite[Section 3.5]{Kato})
 $$
 \RR_{\LLL_0}(z) = \sum_{k=-1}^\infty z^{k}\CC_k,\qquad  \CC_k\in\BBB(L^2(m)),
$$
which on a small disc around 0 converges. Thus, there is some $C^0>0$ such that $\|\RR_{\LLL_0}(z)\|_{\BBB(L^2(m))}\leq C^0\,|z|^{-1}$ for any $z\in\Delta_{\alpha}$, $z\neq0$. Finally, we notice that
$$
 \min(1,\alpha-\bar\alpha)\|f\|_{H^1_v(m)} \,\,\leq\,\,(1+ NC^0|z|^{-1})\,\|g\|_{L^2(m^2)},
$$
therefore, it suffices to take $C_{H^1_v}=1+/ NC^0 \min(1,\alpha-\bar\alpha)$, with $N$ large enough.
\end{proof}


\section{Stability of the stationary solution in the small connectivity regime} 
\label{sec:StabilityStatSol}
\setcounter{equation}{0}
\setcounter{theo}{0}

Now, we establish the exponential convergence of the nonlinear equation. To that aim, we first notice that, in the small connectivity regime, the linear operator $\LLL_\eps$ inherits (in a sense that we precise later on) the stability properties of $\LLL_0$.
\subsection{Uniqueness of the stationary solution in the weak connectivity regime}

As a first step in the proof of Theorem~\ref{th:WeakCregime}, we need a uniqueness condition that, for instance, can be settled as a consequence of the following estimate:

%
%
 \begin{lem}\label{lem:EstimL0perp} There exists a constant $C_\VV$ such that for any $g \in L^2(m), \,\, \langle g \rangle = 0 $ and for the solution $f \in L^2(m)$ to the
 linear equation $\LLL_0 f = g$ there holds
 \beqn\label{eq:EstimL0perp} 
 \|f \|_{\VV} := \| f \|_{L^2(Mm)} + \| \nabla_v f \|_{L^2(M^{1/2}m)} \le C_\VV \, \| g \|_{L^2(m)}. 
 \eeqn
 \end{lem}

\begin{proof} We easily compute
\begin{eqnarray}
\nonumber
 \int_{\R^2} (\LLL_0f)  f\,Mm^2 \,=\,-\int_{\R^2}p(x,v)f^2\,m^2-\int_{\R^2}(\partial_vf)^2\,Mm^2,
\end{eqnarray}
 for some $p(x,v)$ polynomial in $x$ and $v$ with leading term $v^6+x^4$. Therefore, there exists some constants $K_1>0$ and $0<K_2<1$, such that
$$
 \int_{\R^2} (\LLL_0 f) f Mm^2 \,\le\, K_1 \int_{\R^2} f^2  m^2  - K_2 \int_{\R^2} f^2  M^2m^2 - K_2  \int_{\R^2} (\partial_v f)^2  Mm^2 .
 $$
 
 The invertibility of $\LLL_0$ in $L^2(m)$ for zero mean functions, writes
 $$
 \LLL_0 f = g \in L^2(m), \,\, \langle g \rangle = 0 
 \quad\Rightarrow\quad 
 \| f \|_{L^2(m)} \le C_{\bar\alpha}\, \| g \|_{L^2(m)},
 $$
with $C_{\bar\alpha}$ given in the proof of the stability part of Theorem~\ref{th:StatExist}. As a consequence, for any $f$ and $g$ as in the statement of the lemma, we have
 \bean
&&  \int_{\R^2} f^2  M^2m^2  + \int_{\R^2} (\partial_v f)^2  Mm^2 
  \,\le\, -\frac{1}{K_2}  \int_{\R^2} g\, f  Mm^2 + \frac{K_1}{K_2}  \int_{\R^2} f^2  m^2
\\
&&\qquad \,\le\,  {1 \over 2} \int_{\R^2} f^2   \, M^2m^2 +  \frac{1}{2K_2^2}  \int_{\R^2} g^2  \, m^2 + \frac{K_1C}{K_2} \int_{\R^2} g^2  \, m^2, 
 \eean
 from which \eqref{eq:EstimL0perp}  immediately follows.
\end{proof}

\begin{cor}\label{prop:WeakCregimeU} There exists $\eps_1\in(0,\eps_0)$ such that in the small connectivity regime $\eps \in (0,\eps_1)$ the stationary solution is unique.
\end{cor}

\begin{proof}
We write 
\bear
G_\eps - F_\eps 
&=& \eps\, \LLL_0^{-1} \Big[ \partial_v \Big(   ( v- \JJJ(F_\eps)) F_\eps -  ( v- \JJJ(G_\eps) ) G_\eps \Big)  \Big] 
      \nonumber \\
\label{eq:G-F}
&=& \eps\, \LLL_0^{-1} \Big[  \partial_v \Big( ( v- \JJJ(F_\eps)) (F_\eps - G_\eps) +  ( \JJJ(F_\eps) -  \JJJ(G_\eps) ) G_\eps \Big) \Big] .
\eear
As a consequence, using the invertibility property  of $\LLL_0$ for zero mean functions, and the uniform bound~\eqref{eq:H1m} on $G_\eps$, $F_\eps$, we get
\bean
 \| F_\eps - G_\eps \|_{\VV} 
&\le& \eps \, C_{\bar\alpha} \,\big \|  \partial_v \big( ( v- \JJJ(F_\eps)) (F_\eps - G_\eps) +  ( \JJJ(F_\eps) -  \JJJ(G_\eps) ) G_\eps \big) \big \|_{L^2(m)}
\\
&\le&  \eps \, C \, \| F_\eps - G_\eps \|_{\VV},
\eean
for some $C$ depending on the parameters of the system and $\eps_0$. The previous relationship implies, in particular, that $\|F_\eps - G_\eps \|_{\VV} = 0$ for $\eps<\eps_1=1/C$. 
\end{proof}
  
 \subsection{Study of the Spectrum and Semigroup for the Linear Problem}
 
We now turn into a generalisation of Proposition~\ref{prop:L0Ext} in the case $\eps>0$ small. Since the positivity of the operator is lost, Krein-Rutman theory does not apply anymore, however we can prove the following result based on a perturbation argument

\begin{theo}\label{theo:LepsExt} Let us fix $\alpha\in(\bar\alpha,0)$. Then there exists $\eps_2\in(0,\eps_1)$ such that for any $\eps\in[0,\eps_2]$, there hold
\begin{itemize}
 \item[(i)] The spectrum $\Sigma(\LLL_\eps)$ of $\LLL_\eps$ in $L^2(m)$ writes 
 $$
 \Sigma(\LLL_\eps)\cup \Delta_{\alpha}\,=\,\{\mu_\eps\},
 $$
  where $\mu_\eps$ is a eigenvalue simple. Moreover, since $\LLL_\eps$ remains in divergence form, we still have $$\LLL_\eps^*1=0$$ and then $\mu_\eps=0$.
 \item[(ii)] The linear semigroup $S_{\LLL_\eps}(t)$ associated to $\LLL_\eps$ in $L^2(m)$ writes $$S_{\LLL_\eps}(t)=e^{\mu_\eps t}\Pi_\eps+ R_\eps(t),$$ where $\Pi_\eps$ is the projection on the eigenspace associated to $\mu_\eps$ and where $R_\eps(t)$ is a semigroup which satisfies $$\|R_\eps(t)\|_{\BBB(L^2(m))}\leq C_{\LLL_{\eps_1}}\,e^{\alpha t},$$ for some positive constant $C_{\LLL_{\eps_1}}$ independent of $\eps$.
\end{itemize}
\end{theo}

\bigskip
To enlighten the key points of the proof we present it in three steps: \textit{accurate preliminaries, geometry of the spectrum} of the linear operator in the small connectivity regime and \textit{sharp study of the spectrum close to 0:}
\bigskip

\noindent
  {\sl Step 1. Accurate preliminaries:} Let us introduce the operator 
  $$
  P_\eps=\LLL_\eps-\LLL_0 = -\,\eps\,\partial_v((v-\JJJ(G_\eps))\,\cdot)+\eps\,\JJJ(\cdot)\,\partial_v G_\eps.
  $$
Our aim is to estimate the convergence to 0 of this operator in a suitable norm. We notice that, for two exponential weights $m_1,m_2$ as in~\eqref{eq:weightFunc} with $\kappa_1<\kappa_2$, it holds
\begin{eqnarray*}
  \|P_\eps h\|_{L^2(m_1)}^2 &\leq& C\,\eps^2\int_{\R^2}\big( h^2+ v^2|\partial_vh|^2\big)\, m_1^2+C\,\eps^2\JJJ(h)^2
  \\
  &\leq& C\,\eps^2\big(\|h\|_{L^2(m_1)}^2+\|\partial_v h\|_{L^2(m_2)}^2\big),
\end{eqnarray*}
where $C$ depends only on the parameters of the system and, in the small connectivity regime, on $\eps_1$. Therefore, there exists $C_{P_{\eps_1}}>0$ such that
$$
  \|P_\eps h\|_{L^2(m_1)} \leq C_{P_{\eps_1}}\,\eps\|h\|_{H^1_v(m_2)}.
$$

\bigskip

\noindent{\sl Step 2. Geometry of the spectrum of $\LLL_\eps$.} 

\begin{lem}
 For any $z\,\in\,\Delta_{\alpha}$, $z\neq0$ let us define $K_\eps(z)$ by
 $$
  K_\eps(z)=-\,P_\eps\, \RR_{\LLL_0}(z)\,\AA \RR_{\BB_\eps}(z).
 $$
 Then, there exists $\eta_2(\eps)\xrightarrow[\eps\rightarrow0]{}0$, such that
 $$
  \forall\,z\,\in\,\Omega_{\eps}\,:=\,\Delta_{\alpha}\setminus \bar B(0,\eta_2(\eps)),\quad\|K_\eps(z)\|_{\BBB(L^2(m))}\leq \eta_2(\eps)(1+\eta_2(\eps)).
 $$
 Moreover, there exists $\eps_2\in(0,\eps_1]$ such that for any $\eps\in[0,\eps_2]$ we have
 \begin{enumerate}
  \item $I+K_\eps(z)$ is invertible for any $z\in\Omega_\eps$
  \item $\LLL_\eps-z$ is also invertible for any $z\in\Omega_\eps$ and
  $$
   \forall\,z\in\Omega_\eps,\quad \RR_{\LLL_\eps}(z)=\UU_\eps(z)\big(I+K_\eps(z)\big)^{-1}
   $$
   where
   $$
   \UU_\eps(z)=\RR_{\BB_\eps}(z)-\RR_{\LLL_0}(z)\,\AA\,\RR_{\BB_\eps}(z).
  $$
 \end{enumerate}
 We thus deduce that
 $$
  \Sigma(\LLL_\eps)\cap\Delta_{\alpha}\subset B(0,\eta_2(\eps)).
 $$
\end{lem}

\begin{proof}
 We define $m_1$ and $m_2$ two exponential weights with $m_1=m$. From Lemma~\ref{lemma:Bdiss}, Proposition~\ref{prop:L0Ext} and the Step 1 we get that for any $z\in\Omega_\eps$, any $h\in L^2(m)$
\begin{eqnarray*}
 \|K_\eps(z) h\|_{L^2(m)} &\leq& \eps\,C_{P_{\eps_1}}\|\RR_{\LLL_0}(z)\AA\,\RR_{\BB_\eps}(z)h\|_{H^1_v(m_2)}
 \\
 &\leq& \eps\, C_{P_{\eps_1}}\,C_{H^1_v}(1+|z|^{-1})\|\AA\,\RR_{\BB_\eps}(z)h\|_{L^2(m_2)}
 \\
 &\leq& \eps\, C_{P_{\eps_1}}\,C_{H^1_v}(1+|z|^{-1})\, C_{\eps_1}\|h\|_{L^2(m)},
\end{eqnarray*}
where $C_{\eps_1}$ is an upper bound of $\|\AA\RR_{\BB_\eps}\|_{\BBB(L^2(m),L^2(m_2))}$ and do not depend on $\eps$.
Defining $$\eta_2(\eps):=(\eps\,C_{P_{\eps_1}}\,C_{H^1_v}\,C_{\eps_1})^{1/2},$$ it holds
$$
  \|K_\eps(z)\|_{\BBB(L^2(m))}\leq \eta_2(\eps)^2(1+\eta_2(\eps)^{-1})=\eta_2(\eps)(1+\eta_2(\eps)),\quad  \forall\,z\,\in\Omega_\eps,
$$
therefore, fixing $\eps_2>0$ such that $$\eta_2(\eps)<1/2,\quad \forall\,\eps\in(0,\eps_2],$$ we obtain the invertibility of $I+K_\eps(z)$.
\smallskip

Finally, for any $z\,\in\,\Omega_\eps$:
$$
  (\LLL_\eps-z)\,\UU_\eps(z )= I+K_\eps(z),
$$
then there exists a right inverse of $\LLL_\eps-z$. The rest of the proof is similar to the proof of~\cite[Lemma 2.16]{Tristani}.
 \end{proof}

%
%
%
\noindent
{\sl Step 3. Sharp study of spectrum close to 0.} 


%

Let us fix $r\in(0,-\alpha]$ and choose any $\eps_r\in[0,\eps_2]$ such that $\eta_2(\eps_r)<r$ in such a way that $\Sigma(\LLL_\eps)\cap\Delta_{\alpha}\subset B(0,r)$ for any $\eps\in[0,\eps_r]$.  We may define the spectral projection operator
$$
 \Pi_\eps:=-\frac1{2\pi i}\int_{|z'|=r}\RR_{\LLL_\eps}(z')\,dz'.
$$
We have then the
\begin{lem}\label{lemma:?}
 The operator $\Pi_\eps$ is well defined and bounded in $L^2(m)$. Moreover, for any $\eps\in[0,\eps_r]$, it holds
 $$
  \|\Pi_{\eps}-\Pi_{0}\|_{\BBB(L^2(m))}\leq \eta_3(\eps),
 $$
 for some $\eta_3(\eps)\xrightarrow[\eps\rightarrow0]{}0$.
\end{lem}
\begin{proof}
 Let us notice that
\begin{equation}\nonumber
 \Pi_0\,\,=\,\, -\frac{1}{2\pi i}\int_{|z'|=r}(\RR_{\BB_0}(z')-\RR_{\LLL_0}\,\AA\,\RR_{\BB_0}(z'))\,dz'\,\,=\,\,\frac{1}{2\pi i}\int_{|z'|=r}\RR_{\LLL_0}\,\AA\,\RR_{\BB_0}(z')\,dz'
\end{equation}
and
\begin{eqnarray*}
 \Pi_\eps &=& -\frac{1}{2\pi i}\int_{|z'|=r}(\RR_{\BB_\eps}(z')-\RR_{\LLL_0}\,\AA\,\RR_{\BB_\eps}(z'))(I+K_\eps(z'))^{-1}\,dz'
 \\
 &=&\frac{1}{2\pi i}\int_{|z'|=r}\RR_{\BB_\eps}(z')\,K_\eps(z')(I+K_\eps(z'))^{-1}\,dz'
 \\
 &&\quad+\frac{1}{2\pi i}\int_{|z'|=r}\RR_{\LLL_0}\,\AA\,\RR_{\BB_\eps}(z')(I+K_\eps(z'))^{-1}\,dz'.
\end{eqnarray*}
Then, we deduce that
\begin{eqnarray*}
 \Pi_\eps-\Pi_0 &=& \frac{1}{2\pi i}\int_{|z'|=r}\RR_{\BB_\eps}(z')\,K_\eps(z')(I+K_\eps(z'))^{-1}\,dz'
 \\
 &&\quad+\frac{1}{2\pi i}\int_{|z'|=r}\RR_{\LLL_0}\,\AA\,(\RR_{\BB_\eps}(z')-\RR_{\BB_0}(z'))\,dz'
 \\
 &&\qquad+\frac{1}{2\pi i}\int_{|z'|=r}\RR_{\LLL_0}\,\AA\,\RR_{\BB_\eps}(z')(I-(I+K_\eps(z'))^{-1})\,dz',
\end{eqnarray*}
here, the first and third terms are going to 0 because of the upper bounds of $K_\eps(z)$. For the second term, it suffices to notice that
\begin{equation}\nonumber
\RR_{\BB_\eps}(z')-\RR_{\BB_0}(z') \,\,=\,\, \RR_{\BB_0}(z')\,(\BB_\eps-\BB_0)\,\RR_{\BB_\eps}(z'),
\end{equation}
and use that $(\BB_\eps-\BB_0)=P_\eps$.
\end{proof}

To conclude the proof we recall the following lemma from~\cite[paragraph I.4.6]{Kato}
\begin{lem}
 Let $X$ be a Banach space and $P,Q$ two projectors in $\BBB(X)$ such that $\|P-Q\|_{\BBB(X)}<1$. Then the ranges of $P$ and $Q$ are isomorphic. In particular, $\text{dim}(R(P))=\text{dim}(R(Q))$.
\end{lem}
\noindent Provided with this lemma and fixing $\eps'$ such that $\eta_3(\eps')<1$, we get the
\begin{cor}
 There exists $\eps'>0$ such that for any $\eps\in[0,\eps']$ there holds
 $$
  \Sigma(\LLL_\eps)\cap\Delta_{\alpha}=\{\mu_\eps\}\quad\text{and the eigenspace associated to }\mu_\eps\text{ is 1-dimensional.}
 $$
\end{cor}
 
 \subsection{Exponential stability of the NL equation}
 
 In the small connectivity regime $\eps\in(0,\eps')$, let us consider the variation $h:=f_\eps-G_\eps$, with $f_\eps$ the solution to~\eqref{eq:FhNs} and $G_\eps$ the unique solution to~\eqref{eq:StatExist} given by Theorem~\ref{th:StatExist}. By definition, $h$ satisfies the evolution PDE:
\begin{equation}\nonumber
 \partial_t h \,=\, \LLL_0 h -\eps\partial_v(vh)+\eps\JJJ(f_\eps)\partial_vf_\eps-\eps\JJJ(G_\eps)\partial_vG_\eps 
 \,=\, \LLL_\eps h+\eps\JJJ(h)\partial_v h,
\end{equation}
moreover, the nonlinear part is such that
\begin{eqnarray*}
 \|\eps\JJJ(h)\partial_v h\|_{L^2(m)}&\leq& C\,\eps\,\|h\|_{L^2(m)}\|\partial_vh\|_{L^2(m)}
\end{eqnarray*}
for some positive constant $C$.

\begin{proof}[Proof of Theorem~\ref{th:WeakCregime}]
Let us first notice that, thanks to inequality~\eqref{eq:H1m} and the definition of $\JJJ(\cdot)$, we have that
$$
 \|\eps\JJJ(h)\partial_v h\|_{L^2(m)}\leq C_{NL}\,\eps\,\|h\|_{L^2(m)},\qquad \forall\, h_0\in H^1(m),
$$
where
$$
 C_{NL}\,\,=\,\, c_\delta^{-1}\max(C_2,\|h_0\|_{H^1(m)}).
$$
On the other hand, Duhamel's formula reads
$$
 h=S_{\LLL_\eps}(t)h_0+\int_0^t S_{\LLL_\eps}(t-s)\big(\eps\JJJ(h)\partial_v h\big)\,ds,
$$
then, we have that
\begin{eqnarray*}
 u(t)\,\,:=\,\,\|h\|_{L^2(m)} &\leq& \|S_{\LLL_\eps}(t)h_0\|_{L^2(m)}+\int_0^t\|S_{\LLL_\eps}(t-s)\big(\eps\JJJ(h)\partial_v h\big)\|_{L^2(m)}\,ds\\
 &\leq& C_{\LLL_{\eps_1}}\,e^{\alpha t}\|h_0\|_{L^2(m)}+C_{\LLL_{\eps_1}}C_{NL}\,\eps\,\int_0^t e^{\alpha(t-s)}\|h\|_{L^2(m)}\,ds
 \\
 &=&C_{\LLL_{\eps_1}}\,e^{\alpha t}u(0)+C_{\LLL_{\eps_1}}C_{NL}\,\eps\,\int_0^t e^{\alpha(t-s)}u(s)\,ds.
\end{eqnarray*}
In particular, 
$$ 
 u(t)\,\leq\, C_{\LLL_1}\,u(0)\,e^{(\alpha+C_{\LLL_{\eps_1}}C_{NL}\eps)t},
$$
Summarising, it suffices to define $\eta^*(\eps):= C_2/\sqrt{\eps}$ to get that for any $f_0$ such that
$$
 \|f_0-G_\eps\|_{H^1(m)}\,\,\leq\,\,\eta^*(\eps),
$$
it holds
$$
 \|f_\eps(t)-G_\eps\|_{L^2(m)}\,\leq\, C_{\LLL_{\eps_1}} \|f_0-G_\eps\|_{L^2(m)} e^{\alpha^* t},
$$
with 
$$
\alpha^*\,\,=\,\,\alpha+C_{\LLL_{\eps_1}}c_{\delta}^{-1}C_2\sqrt{\eps^*}<0,
$$
if $\eps^*$ is small enough.
\end{proof}

\section{Open problems beyond the weak coupling regime}
\label{sec:Beyond}
\setcounter{equation}{0}
\setcounter{theo}{0}

In the weak coupling regime, we have demonstrated that existence and uniqueness of solutions persist. In that regime, noise overcomes nonlinear effects and the system is mixing: one finds a unique distribution with an everywhere strictly positive density. As coupling increases, highly non-trivial phenomena may emerge as nonlinear effects of the McKean-Vlasov equation. For instance, it is likely that in another asymptotic regime in which coupling is non-trivial and noise goes to zero, Dirac-delta distributed solutions shall emerge (in which all neurons are synchronized and their voltage and adaptation variable are equal to one of the stable fixed point of the deterministic Fitzhugh-Nagumo ODE). 

Here, we numerically explore the dynamics of the Fitzhugh-Nagumo McKean-Vlasov equation using a Monte-Carlo algorithm. We observe that complex phenomena occur as the coupling is varied. That numerical evidence tends to show that several additional equilibria may emerge, the stability of stationary solutions may change as a function of connectivity levels, and attractive periodic solution in time may emerge. These regimes are particularly interesting from the application viewpoint: indeed, among important collective effects in biology, from large networks often emerge bistable high-state of down-states (characterized by high or low firing rates), and even oscillations. These two phenomena are particularly important in developing and storing memories, and this occurs by slowly reinforcing connections~\cite{kandel-schwartz-etal:00}. Interestingly, these two types of behaviors emerge naturally in the FhN McKean Vlasov equation beyond weak coupling. For instance, for fixed $\sigma=0.5$, we present the solutions of the particle system varying the connectivity weight beyond small values, both in the bistable case (in which the  FhN model presents two stable attractors) and the excitable regime, the most relevant for biological applications, characterized by a single stable equilibrium and a manifold separating those trajectories doing large excursions (spikes) from those returning to the resting state directly. In both cases, we observe (i) that the unique stationary solution is not centered close from a fixed point of the dynamical system: neurons intermittently fire in an asynchronous manner for small coupling. As coupling increases, a periodic attractive solution emerges, before the appearance of distinct stationary solutions (two in the bistable case, one in the excitable case). These phenomena are depicted in Fig.~\ref{fig:beyond}. Proving, for larger coupling, the existence and stability of a periodic solution or distinct and multiple stationary solutions constitute exciting perspectives of this work. 

\begin{figure}
	\centering
		\includegraphics[width=.7\textwidth]{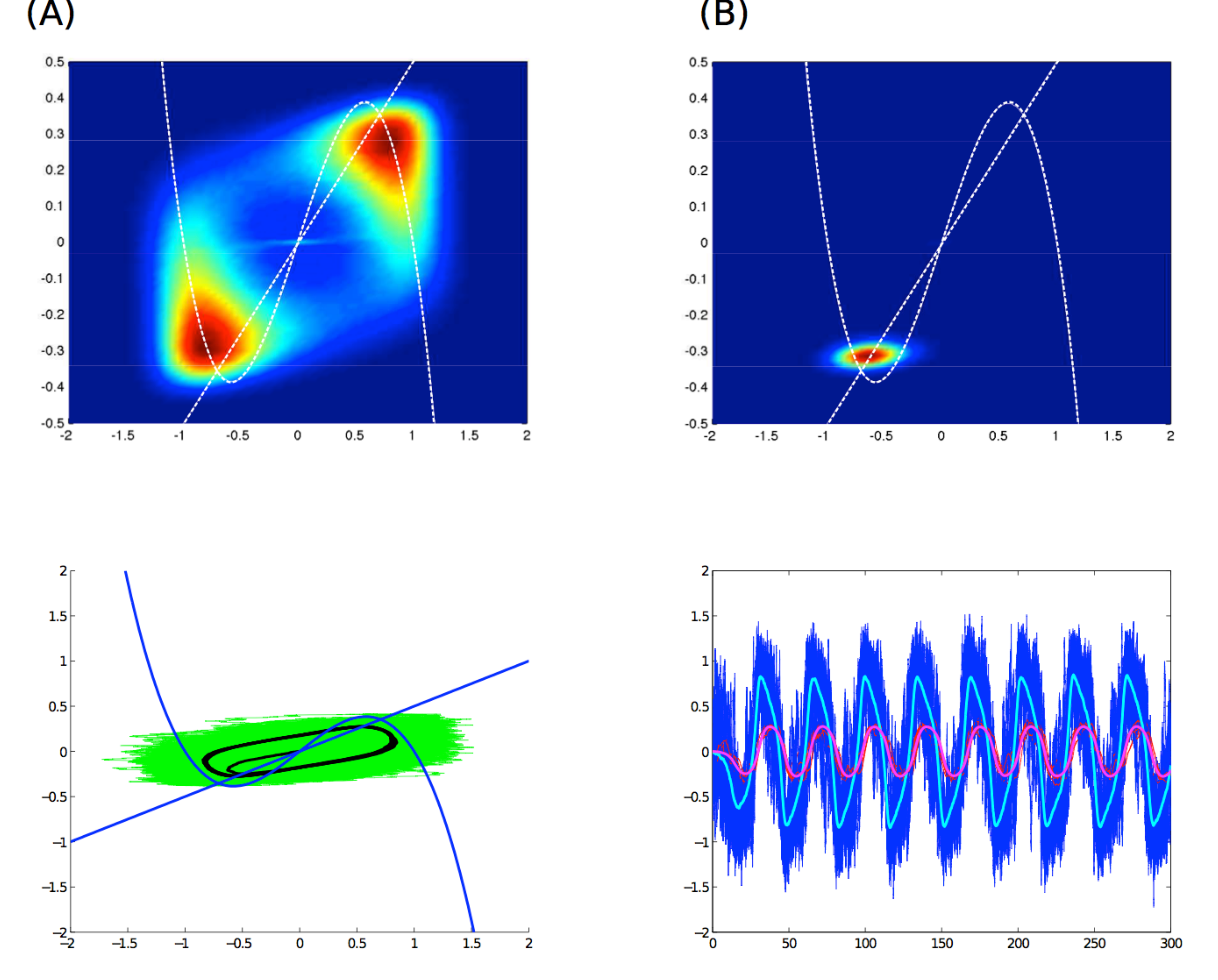}
	\caption{Permanent (non-transient) regimes of the FhN particle system for $N=2\,000$. Top row: $J=0.1$ (A) and $J=3$ (B), bottom row: $J=1$. The unique stationary solution in the small coupling limit  analyzed in the manuscript visits both attractors transiently (A), while in the high coupling regime (B), the system remains around one of the attractors (the system has at least two such solutions). In an intermediate regime, the system shows periodic oscillations (bottom row). }
	\label{fig:beyond}
\end{figure}

These phenomena are actually conjectured to be generic in coupled excitable systems subject to noise. 

\appendix

\section{Mean-Field limit for Fitzhugh-Nagumo neurons}
\label{app:MFLim}
\setcounter{equation}{0}
\setcounter{theo}{0}

Let us start by a well known result with is a simple application of global existence and path wise uniqueness for system of SDE, see~\cite[Chapter 5, Theorems 3.7 and 3.11]{ethier-et-kurtz:86} for example. Consider the particle system for $1\leq i\leq N$:
\begin{equation}\label{eq:IBM}
 \begin{cases}
  \displaystyle dv^i_t = \big(v_t^i\,(v_t^i-\lambda)\,(1-v_t^i)-x_t^i+I_0\big)\,dt+\frac JN\sum_{j=1}^N \big(v_t^i-v_t^j\big)\,dt+ dW^i_t \\
    \displaystyle dx^i_t = (-a x_t^i+bv^i_t)dt,
 \end{cases}
\end{equation}
with initial data $(X_0^i,V_0^i)$ for $1\leq i\leq N$ distributed according to $f_0\in\Pp_2(\R^2)$, i.e., a probability measure in $\R^2$ with finite second moment. Here the $(W_t^i)_{t\geq0}$ are $n$ independent standard Brownian motions in $\R$. This result was stated in~\cite{BFFT}. In that paper, the authors use a stopping in the $n$-voltage variables which requires finely controlling all trajectories. We prove here a simpler version of the result based on \emph{a-priori} estimates.

\begin{lem}\label{lemme:IBM}
 Let $f_0\in\Pp(\R^2)$ be a probability with finite second moment, and a set of random variables $(X_0^i,V_0^i)$ with law $f_0$. Then~\eqref{eq:IBM} admits a path wise unique global solution with initial datum $(X_0^i,V_0^i)$ for $1\leq i\leq N$. 
\end{lem}

\begin{proof}
The system~\eqref{eq:IBM} can be written in $\R^{2N}$ as the SDE
\begin{equation}\nonumber
d\mathbf{Z}_t^N = \sigma^N\,d\mathbf{B}_t^N+\mathbf{b}(\mathbf{Z}_t^N)\,dt,
\end{equation}
where $\mathbf{Z}_t^N=(x_t^1,v_t^1,\ldots,x_t^N,v_t^N)$, $\sigma^N$ is a constant $2N\times 2N$ sparse matrix, $(\mathbf{B}_t^N)_{t\geq0}$ is a standard Brownian motion on $\R^{2N}$, and $\mathbf b:\R^{2N}\rightarrow\R^{2N}$ is a function defined in the obvious way. It is easy to see that $\mathbf{b}$ is a locally Lipschitz function, moreover, letting $\langle\cdot,\cdot\rangle$ and $\|\cdot\|$ the scalar product and the Euclidean norm on $\R^{2N}$ respectively, then for all $\mathbf{Z}^N=(x^1,v^1,\ldots,x^N,v^N)$,
\begin{eqnarray*}
 \langle\mathbf{Z}^N,\mathbf{b}(\mathbf{Z}^N)\rangle &=& \sum_{i=1}^N x^i(-ax^i+bv^i) + \sum_{i=1}^N v^i\big(v^i\,(v^i-\lambda)\,(1-v_t^i)-x^i+I_0\big)+\frac JN\sum_{i.j=1}^N v^i\big(v^i-v^j\big)\\
 &\leq& \sum_{i=1}^N (b-1)x^iv^i + \sum_{i=1}^N \big(J|v^i|^2-a|x^i|^2\big)-\frac JN\sum_{i.j=1}^N v^iv^j+CN\\
 &\leq& C(1+\|\mathbf{Z}^N\|^2).
\end{eqnarray*}
This is a sufficient condition for global existence and pathwise uniqueness (see e.g.~\cite{Mao2007}).
\end{proof}

\subsection*{Mean-Field limit}

Now we turn to the propagation of chaos property. We already know the existence and uniqueness of the particle system~\eqref{eq:IBM}, moreover the nonlinear SDE:
\begin{equation}\label{eq:nonLinearSDE}
 \begin{cases}
  \displaystyle d\bar v_t =  \big(\bar v_t(\bar v_t-\lambda)(1-\bar v_t)-\bar x_t+I\big)\,dt+J\int_{\R^2}(\bar v_t-v)\,df_t(x,v)\,dt+ dW_t,\\
  \displaystyle d\bar x_t = (-a\bar x+b\bar v_t)dt\\
  \displaystyle f_t=\text{law}(\bar x_t,\bar v_t),\quad\text{law}(\bar x_0,\bar v_0)=f_0.
 \end{cases}
\end{equation}
is also well-posed for $f_0\in L^1(M^2)\cap L^1\log L^1\cap\Pp(\R^2)$, as a consequence of Theorem~\ref{th:E&U&B}. Then, for instance, we can sate the

\begin{theo}\label{theorem:PropagationOfChaos}
Let $f_0$ be a Borel probability measure and $(X_0^i,V_0^i)$ for $1\leq i\leq N$ be $N$ independent variables with law $f_0$. Let us assume that the solutions to~\eqref{eq:IBM} and~\eqref{eq:nonLinearSDE} with initial data $(X_0^i,V_0^i)$ and $f_0$ are well defined on $[0,T]$ and such that
\begin{equation}\label{eq:ChaosCondition1}
\sup_{[0,T]}\Big\{\int_{\R^2}(|x|^2+|v|^2)\,df_t(x,v)\Big\}<+\infty,
\end{equation}
with $f_t=\text{law}(\bar x_t^i,\bar v_t^i)$ (which actually does not depend on $i$ by exchangeability). Then there exists a constant $C>0$ such that
\begin{equation}\label{eq:ChaosProp1}
 \E\big[|x_t^i-\bar x_t^i|^2+|v_t^i-\bar v_t^i|^2\big]\leq\frac{C}{N}e^{Ct}.
\end{equation}
\end{theo}

\begin{proof}
We start by writing $X^i_t=x^i_t-\bar x^i_t$ and $V^i_t=v^i_t-\bar v^i_t$. For notational convenience we drop the time dependence subindex and take $J=1$. Because $x^i_t$ and $\bar x^i_t$ are driven by the same Brownian motion, we have that
\begin{equation}\nonumber
 \begin{cases}
  \displaystyle d V^i = \big(v^i(v^i-\lambda)(1-v^i)-\bar v^i(\bar v^i-\lambda)(1-\bar v^i)- X^i\big)\,dt+\frac 1N\sum_{j=1}^N \big(v_t^i-v_t^j\big)\,dt-\int_{\R^2}(\bar v^i-v)\,df_t(x,v)\,dt\\ 
  \displaystyle d X^i = (-aX^i+b V^i)dt,
\end{cases}
\end{equation}

We define $\alpha(t)=\E\big[|X^i|^2+|V^i|^2\big]$ which is independent of the label $i$ by symmetry and exchangeability of the system. It is not hard to see that
\begin{equation}\nonumber
 \frac12\frac d{dt}\E\,\big[|X^i|^2\big]=\E\big[b|X^i|\,|V^i|-a|X^i|^2 \big]\leq \frac b2\alpha(t),
\end{equation}
and
\begin{eqnarray*}
 \frac12\frac d{dt}\E\,\big[|V^i|^2\big]&=&\E\,\big[V^i\big(v^i(v^i-\lambda)(1-v^i)-\bar v^i(\bar v^i-\lambda)(1-\bar v^i) - X^i\big)\big]\\
 &&\quad+\,\E\,\Big[\frac {V^i}N\sum_{j=1}^N \big(v_t^i-v_t^j\big)\,dt-V^i\int_{\R^2}(\bar v^i-v)\,df_t(x,v)\Big]=:\SS_1+\SS_2.
\end{eqnarray*}

\noindent\textit{Estimate for $\SS_1$:} Let us first notice that 
  \begin{eqnarray*}
 v^i(v^i-\lambda)(1-v^i)-\bar v^i(\bar v^i-\lambda)(1-\bar v^i)&=& -(|v^i|^3-|\bar v^i|^3)+(1+\lambda)(|v^i|^2-|\bar v^i|^2)-\lambda V^i \\
 &=& -V^i(|v^i|^2+v^i\,\bar v^i+|\bar v^i|^2)+(1+\lambda)V^i(|v^i|+|\bar v^i|)-\lambda V^i,
  \end{eqnarray*}
  therefore
  $$
   \SS_1 \,\,=\,\, \E[|V^i|^2(-|v^i|^2-v^i\,\bar v^i-|\bar v^i|^2+(1+\lambda)(|v^i|+|\bar v^i|)-\lambda)]-\E[V^iX^i],
  $$
 and by consequence there is some constant $C>0$ such that
  \begin{equation}\label{eq:S1bound}
 \SS_1\,\, \leq \,\, C\alpha(t).
 \end{equation}

\noindent \textit{Estimate for $\SS_2$:} By definition, it holds
 \begin{eqnarray*}
 \SS_2 &=& \E\,\Big[V^i(v_t^i-\bar v^i)-\frac {V^i}N\sum_{j=1}^N \big(v_t^j-\int_{\R^2}v\,df_t(x,v)\big)\Big]\\
 &=&\E\,\big[|V^i|^2\big]-\frac 1N\,\E\,\,\Big[V^i\sum_{j=1}^N \Big(v_t^j-\int_{\R^2}v\,df_t(x,v)\Big)\Big].
 \end{eqnarray*}
 Moreover, by symmetry we know that $\SS_2$ does not depend on a particular $i$, therefore we take $i=1$ to get
 \begin{equation}\nonumber
  \SS_2\leq\E\,\big[|V^1|^2\big]+\frac{1}{N}\Big(\E\,\big[|V^1|^2\big]\Big)^{1/2}\Big(\E\,\Big[\big|\sum_{j=2}^N \Big(v_t^j-\int_{\R^2}v\,df_t(x,v)\Big)\big|^2\Big]\Big)^{1/2}.
 \end{equation}
Now, defining $Y^j=v_t^j-\int_{\R^2}v\,df_t(x,v)$, for $j\neq k$, we find that
 \begin{equation}\nonumber
 \E\big[Y^j Y^k\big]=\E\Big[\E\big[Y^j\mid(\bar x^1,\bar v^1)\big]\E\big[Y^k\mid(\bar x^1,\bar v^1)\big]\Big],
\end{equation}
 but
 \begin{equation}\nonumber
  \E\big[Y^j\mid(\bar x^1,\bar v^1)\big]=\E\Big[v_t^j-\int_{\R^2}v\,df_t(x,v)\Big]=0.
 \end{equation}
 Hence, fixing $j_\ast\in\{2,\ldots,N\}$
 \begin{eqnarray*}
 \E\Big[\big|\sum_{j=2}^N \Big(v_t^j-\int_{\R^2}v\,df_t(x,v)\Big)\big|^2\Big]&=&(N-1) \E\Big[\big|v_t^{j_\ast}-\int_{\R^2}v\,df_t(x,v)\Big)\big|^2\Big]\\
 &=&(N-1)\int_{\R^2}\Big(w-\int_{\R^2}v\,df_t(x,v)\Big)^2\,df_t(y,w)\leq C(N-1),
 \end{eqnarray*}
 since the second moment of $f_t$ is uniformly bounded in $[0,T]$. Finally we conclude that
 \begin{equation}\label{eq:S2bound}
 \SS_2 \,\, \leq \,\, \alpha(t)+\alpha(t)^{1/2}\frac{C}{\sqrt{N}}.
 \end{equation}
 
 Finally, going back to the bounds on $\alpha(t)$, we put together~\eqref{eq:S1bound} and~\eqref{eq:S2bound} to find
 \begin{equation}\nonumber
  \frac{d}{dt}\alpha(t)\leq C\alpha(t)+2\alpha(t)^{1/2}\frac{C}{\sqrt{N}}\leq C\alpha(t)+\frac{C}{N},
 \end{equation}
and using Gr\"onwal's Lemma,
$$  
 \alpha(t)\leq \left(\alpha(0)+\frac{C}{N}\right)e^{Ct}=\frac{C}{N}e^{Ct}
$$ 
which finishes the proof.
\end{proof}

\section{Strong maximum principle for the linearized operator}
\label{app:MaximumPrinciple}
\setcounter{equation}{0}
\setcounter{theo}{0}

In this final appendix we shall extend the result provided in~\cite[Corollary A.20]{MR2562709} to our framework. These local positivity estimates are classical in hypoelliptic equations and they are a necessary condition for Theorem~\ref{th:StatExist}. Here, our result is time dependant and by consequence more general than it is needed in the applications.

In the sequel, we shall use the notation
$$
B_r(x_0,v_0)\,:=\,\{(x,v)\in\R^2\,;\quad |v-v_0|\leq r,\,|x-x_0|\leq r^3\},
$$
and come back to the classical notation $\nabla_{x,v}=D_{x,v}$ and $\partial^2_{vv}=\Delta_v$. Also, we simplify the problem by choosing $a=b=1$, but the proof can be easily extended to the general case.

\begin{theo}\label{th:FPtheo}
 Let $f(t,x,v)$ be a classical nonnegative solution of
 \begin{equation}\label{eq:FPapp}
  \frac{\partial}{\partial t} f-\Delta_v f= A(t,x,v)\,\nabla_vf+B(x,v)\,\nabla_x f +C(t,x,v)\, f
 \end{equation}
 in $[0,T)\times\Omega$, where $\Omega$ is an open subset of $\R^2$, and $A,C:[0,T)\times\R^2$ and bounded continuous functions and $B(x,v)=x-v$. Let $(x_0,v_0)\in\Omega$ and $\bar A$ and $\bar C$ upper bounds of respectively $\|A\|_{L^\infty}$ and $\|C\|_{L^\infty}$.
 
 Then, for any $r,\tau>0$ there are constants $\lambda,\, K>0$, only depending on $\bar A,$ $\bar C$ and $r^2/\tau$ such that the following holds: If $B_{\lambda r}(x_0,v_0)\subset \Omega$, $\tau<\min(1/2,-\log(r^3/2|x_0-v_0|))$ and $f\geq\delta>0$ in $[\tau/2,\tau)\times B_{r}(x_0,v_0)$, then $f\geq K\delta$ in $[\tau/2,\tau)\times B_{2r}(x_0,v_0)$.
\end{theo}

Theorem~\ref{th:FPtheo} implies, via covering arguments in variables $t,x,v$ the

\begin{cor}
 If $f\geq0$ solves~\eqref{eq:FPapp} in $[0,T)\times\Omega$ and $f\geq\delta>0$ in $[0,T)\times B_r(x_0,v_0)$, then for any compact set $K\subset\Omega$ containing $(x_0,v_0)$ and for any $t_0\in(0,T)$, we have $f\geq\delta'>0$ in $[t_0,T)\times K$ where $\delta'$ only depends on $\bar A,\bar C, K, \Omega, x_0,v_0,r,t_0,\delta$.
\end{cor}

\begin{proof}[Proof of Theorem~\ref{th:FPtheo}]

We only explain how to adapt the proof of Theorem A.19 given in~\cite{MR2562709}. Let $g= e^{\bar C t}f(t,x,v)$; then $g\,\geq\, f$ and $\LL\, g\,\geq\,0$ in $(0,T)\times\Omega$, where
$$
 \LL = \partial_t+(v-x)\,\nabla_x-\Delta_v-A(t,x,v)\,\nabla_v.
$$

Next, we construct a particular subsolution for $\LL$. In the sequel, $B_r$ stands for $B_r(x_0,v_0)$ and we define $X_t(x_0,v_0)= v_0+(x_0-v_0)e^{-t}$. 
\bigskip

\noindent{\sl Step 1. Construction of the subsolution.}

For $t\in(0,\tau]$ and $(x,v)\in \Omega\setminus B_r$ let
$$
 P(t,x,v)=\alpha\frac{(v-v_0)^2}{2t}-\frac{\beta}{t^2}(v-v_0)(x-X_t)+\gamma\frac{(x-X_t)^2}{2t^3},
$$
with $\alpha,\beta,\gamma>0$ to be chosen later on. Let further define
$$
 \varphi(t,x,v)=\delta\, e^{-\mu \, P(t,x,v)}-\varepsilon,
$$
where $\mu,\varepsilon>0$ will also be chosen later on. If we assume that $\beta^2< \alpha\,\gamma$, then $P$ is a positive quadratic form in the variables $v-v_0$ and $x-X_t$. Clearly
$$
 \LL\,\varphi = -\mu \, \delta \, e^{-\mu\,P} \mathcal{E}(P),
$$
where
$$
 \mathcal{E}(P)=\partial_t P+(v-x)\,\nabla_xP-\triangle_vP+\mu\,|\nabla_vP|^2-A(t,x,v)\,\nabla_vP.
$$
By straightforward computation we find that $\EE=\EE_1+\EE_2$, with
\begin{eqnarray*}
 \mathcal{E}_1(P) & = & \Big(\mu\,\alpha^2-\frac{\alpha}{2}-\beta\Big)\frac{(v-v_0)^2}{t^2}+2\Big(\beta+\frac\gamma2-\mu\,\alpha\,\beta\Big)\frac{(v-v_0)(x-X_t)}{t^3}
\\
&&\qquad\qquad+\Big(\mu\,\beta^2-\frac{3\,\gamma}{2}\Big)\frac{(x-X_t)^2}{t^4}
\end{eqnarray*}
and
\begin{eqnarray*}
 \mathcal{E}_2(P) & = & \beta\,\frac{(v-v_0)(x-X_t)}{t^2}-\alpha\,\frac{1}{t}
\\
 &&\quad-\gamma\,\frac{(x-X_t)^2}{t^3}-\alpha\,\frac{A(t,x,v)(v-v_0)}{t}+\beta\,\frac{A(t,x,v)(x-X_t)}{t^2}.
\end{eqnarray*}

Now we notice that $\EE_1$ is defined by the quadratic form
$$
 M_q=\begin{bmatrix}
  \displaystyle\mu\,\alpha^2-\frac{\alpha}{2}-\beta && \displaystyle \beta+\frac\gamma2-\mu\,\alpha\,\beta\\
  \displaystyle  \beta+\frac\gamma2-\mu\,\alpha\,\beta &&  \displaystyle \mu\,\beta^2-\frac{3\,\gamma}{2}
 \end{bmatrix}
$$
which is nothing but a quadratic polynomial on $(v-v_0)/t$ and $(x-X_t)/t^2$. As $\mu\rightarrow\infty$
\[\begin{cases}
 \text{tr}\,M_q\,=\,\mu(\alpha^2+\beta^2)+O(1) \\
 \displaystyle \text{det}\,M_q\,=\mu\Big[\,\frac{3\,\alpha\,\beta^2}{2}+\alpha\,\beta\,\gamma-\beta^3-\frac{3\alpha^2\gamma}{2}\Big]+O(1),
\end{cases}\]
both positive quantities if $\beta>\alpha$ and $\alpha\,\gamma>\beta^2$. In particular, for $\beta=2\,\alpha$ and $\gamma=8\,\alpha$,
\[\begin{cases}
 \text{tr}\,M_q\,=5\,\alpha^2\,\mu+O(1) \\
 \displaystyle \text{det}\,M_q\,=2\,\alpha^3\,\mu+O(1),
\end{cases}\]
and letting $\mu\rightarrow\infty$ the eigenvalues of $M_q$ are of order $\mu\,\beta^2$ and $\beta$. So, for any fixed $C>0$ we may choose $\alpha,\beta,\gamma$ and $\mu$ such that
$$
\EE_1(P)\,\,\geq\,\, C\beta\Big(\frac{(v-v_0)^2}{t^2}+\frac{(x-X_t)^2}{t^4}\Big).
$$

Second, if $t\in(0,1)$ then
$$
\EE_2(P)\,\,\geq\,\,-4\beta\frac{(x-X_t)^2}{t^4}-\frac{3\beta(v-v_0)^2}{2}-\frac{3\beta(x-X_t)^2}{2t^4}-2\beta\bar A^2-\frac\beta{2t},
$$
and making $\tau\leq1$, we get,
$$
 \mathcal{E}(P)\geq \text{const}\,\frac{\beta}{t}\Big[C\Big(\frac{(v-v_0)^2}{t}+\frac{(x-X_t)^2}{t^3}\Big)-1\Big]\,,
$$
with $C$ arbitrarily large.

Let us briefly describe the rest of the proof. Recall that $(x,v)\,\notin\,B_r$ so
\begin{enumerate}
 \item either $|v-v_0|\geq r$, then $\mathcal{E}(P)\geq\text{const.}(\beta/t)[Cr^2/\tau-1]$, which is positive for $C>\tau/r^2$;
 \item or $|x-x_0|\geq r^3$, and then, if $\tau\leq\frac12\min(1,-\log(\frac{r^3}{|x_0-v_0|}))$ then for any $t\in[0,\tau)$
 $$
  |X_t-x_0|\leq r^3/2\,\quad\text{and}\quad\frac{|x-X_t|^2}{t^2}\,\geq\, \frac{|x-x_0|^2}{2t^2}-\frac{|X_t-x_0|^2}{t^2}\,\geq\, \frac{r^6}{4\tau^2},
 $$
 so $\mathcal{E}(P)\geq\text{const.}(\beta/t)[Cr^6/4\tau^3-1]$, which is positive as soon as $C>4\tau^3/r^6$.
\end{enumerate}

Summarizing: under the assumptions, we can always choose constants $\gamma>\beta>\alpha>1$ and $\alpha\,\gamma>\beta^2$, depending only on $\bar A$ and $r^2/\tau$, so that 
$$
\LL\,\varphi\geq0,\quad\text{ in }[0,\tau)\times (B_{\lambda r}\setminus B_r),
$$
as soon as $\tau<\min(1/2,-\log(r^3/2|x_0-v_0|))$.
\smallskip

\bigskip
\noindent
{\sl Step 2. Boundary conditions.} We now wish to prove that $\varphi\leq g$ for $t=0$ and for any $(x,v)\in \partial(B_{\lambda r}\setminus B_r)$; then classical maximum principle will do the rest. 

Let us first notice that the boundary condition at $t=0$ is obvious ($\varphi$ can be extended by continuity by $0$ at the initial time). The condition at $\partial B_r$ is also true since $\forall\,(x,v)\in\partial B_r$: $\varphi\leq\delta\leq g$.

It remains to fix the remaining parameters in order to conclude that $\varphi\leq g$ in $\partial B_{\lambda r}$. From the choice of $\alpha,\beta$ and $\gamma$, it is easy to see that for any $(x,v)\in \partial B_{\lambda r}$
:$$
 P(t,x,v)\geq \frac{\alpha}{4}\,\Big(\frac{(v-v_0)^2}{t}+\frac{(x-X_t)^2}{t^3}\Big)\geq \frac{\alpha}4\min\Big(\frac{\lambda^2r^2}{\tau},\frac{\lambda^6r^6}{4\tau^3}\Big)\geq \frac{\alpha\,\lambda^2}{16}\min\Big(\frac{r^2}{\tau},\frac{r^6}{\tau^3}\Big),
$$
notice that we are imposing $\lambda>1$. Choosing
$$
\varepsilon\,=\,\delta\exp\Big( -\frac{\mu\,\alpha\,\lambda^2}{16}\min\Big(\frac{r^2}{\tau},\frac{r^6}{\tau^3}\Big) \Big),
$$ 
we get $\varphi=\delta\,e^{-\mu P(t,x,v)}-\varepsilon\leq 0$ on $\partial B_{\lambda r}$. By consequence $\varphi\leq g$ on the whole set $B_{\lambda r}$.
\smallskip

Let us finally notice that at this point we have uniform bounds for $g$ on $B_{2 r}\setminus B_r$ for any $t\in[\tau/2,\tau)$. Indeed,
$$
 P(t,x,v)\leq 2\,\gamma\,\Bigg(\frac{(v-v_0)^2}{t}+\frac{(x-X_t)^2}{t^3}\Bigg)\leq 2\,\gamma\,\Big(\frac{8\,r^2}{\tau}+\frac{1026\,r^6}{\tau^3}\Big)\leq2068\,\gamma\,\max\Big(\frac{r^2}{\tau},\frac{r^6}{\tau^3}\Big)
$$
Then, for $\lambda$ big enough we find $K_0>0$ such that
$$
 \varphi(t,x,v)\geq \delta\,\Big[\exp\Big(-2068\,\mu\,\gamma\,\max\Big(\frac{r^2}{\tau},\frac{r^6}{\tau^3}\Big)\Big)-\exp\Big( -\frac{\mu\,\alpha\,\lambda^2}{16}\min\Big(\frac{r^2}{\tau},\frac{r^6}{\tau^3}\Big) \Big)\Big]\geq K_0\,\delta,
$$
because $\gamma=8\,\alpha$, to find such $\lambda$ it suffices that
$$
 2068\times16\times8\,\max\Big(\frac{r^2}{\tau},\frac{r^6}{\tau^3}\Big)\leq\lambda^2\,\min\Big(\frac{r^2}{\tau},\frac{r^6}{\tau^3}\Big),
$$
by consequence $\lambda$ depends only on $r^2/\tau$. 

Finally, we find $K,\lambda>0$ depending on $\bar A,$ $\bar C$ and $r^2/\tau$ such that
$$
f\geq K_0\,\delta\, e^{-\tau\,\bar C}\quad\text{on}\quad[\tau/2,\tau)\times (B_{2r}\setminus B_r).
$$
\end{proof}

\begin{rem}
 Let us notice that we can extend Theorem~\ref{th:FPtheo} to some cases when $A$ or $C$ are not necessarily bounded and $\Omega=\R^2$. It suffices to take any $r,\tau>0$ and fix $\lambda$ (which as we saw only depends on a numerical constant and the ratio $r^2/\tau$). We can then fix $R>0$ big enough, in order to have that $\lambda r< R$ and study the equation into $B_R$, where by continuity $A$ and $C$ attain their maximum in the compact set $[0,\tau]\times \bar B_R$.
\end{rem}

\bibliographystyle{acm}

\signsm  \signcq
 \signjt  

\end{document}